\documentclass[12pt]{article}

\usepackage{amsthm}
\usepackage{amsmath}
\usepackage{amssymb}
\usepackage[applemac]{inputenc}

\DeclareMathOperator{\ev}{ev}

\DeclareMathOperator{\cosk}{Cosk}
\DeclareMathOperator{\sk}{Sk}
\DeclareMathOperator{\sh}{\textup{\textbf{Sh}}}

\DeclareMathOperator{\skn}{\sk^n}
\DeclareMathOperator{\tr}{Tr}
\DeclareMathOperator{\trn}{\tr^n}

\def\op{^\textup{\scriptsize{op}}}

\def\elem[#1]{\ulcorner#1\urcorner}
\def\prsh[#1]{\con^{#1^{\textup{\tiny{op}}}}}

\def\tz[#1,#2]{\begin{tikzpicture}[thick,scale=#1]
{#2}
\end{tikzpicture}}

\def\qtri[#1,#2,#3,#4,#5,#6,#7]{
\tz[#1,
\node (1) at (0,.6) {$#2$};
\node (2) at (2,.6) {$#3$};
\node (3) at (2,-.6) {$#4$};
\node at (.8,-.2) {$_{#7}$};
{\draw (1) edge [->] node[above] {$_{#5}$} (2);
\draw (1) edge [->] (3);
\draw (2) edge [->] node[right] {$_{#6}$} (3);}
]
}

\newtheorem{thm}{Theorem}[section]
\def\bt{\begin{thm}}
\def\et{\end{thm}}
\newtheorem{lem}[thm]{Lemma}
\def\bl{\begin{lem}}
\def\el{\end{lem}}
\newtheorem{pro}[thm]{Proposition}
\def\bp{\begin{pro}}
\def\ep{\end{pro}}
\newtheorem{rem}[thm]{Remark}
\def\br{\begin{rem}}
\def\er{\end{rem}}
\newtheorem{de}[thm]{Definition}
\def\bd{\begin{de}}
\def\ed{\end{de}}
\def\pts{\textup{\,:\,}}

\def\unito[#1,#2,#3,#4,#5,#6,#7,#8]{\tz[#1,
\node (1) at (0,.6) {$#2$};
\node (2) at (3,.6) {$#3$};
\node (3) at (1.5,-.6) {$#4$};
\node (a) at (1.6,.6) {};
\node (aa) at (1.6,-.3) {};
{\draw (1) edge[->] node[above]{$_{#5}$} (2); }
{\draw (1) edge[->] node[below left]{$_{#6}$} (3); }
{\draw (3) edge[->] node[below right]{$_{#7}$} (2); }
{\draw (a) edge[->,double,thin] node[left]{$_{#8}$} (aa); }]}

\def\counito[#1,#2,#3,#4,#5,#6,#7,#8]{\tz[#1,
\node (1) at (0,-.6) {$#2$};
\node (2) at (3,-.6) {$#3$};
\node (3) at (1.5,.6) {$#4$};
\node (a) at (1.6,.3) {};
\node (aa) at (1.6,-.6) {};
{\draw (1) edge[->] node[below]{$_{#5}$} (2); }
{\draw (1) edge[->] node[above left]{$_{#6}$} (3); }
{\draw (3) edge[->] node[above right]{$_{#7}$} (2); }
{\draw (a) edge[->,double,thin] node[left]{$_{#8}$} (aa); }]}

\def\con{\textup{\textbf{Set}}}
\def\cate{\textup{\textbf{Cat}}}

\def\bflp{\mathbf{L}_p}

\def\bfc{\mathbf{C}}
\def\bfh{\mathbf{H}}

\usepackage{fullpage}   
\usepackage[all]{xy}

\newcommand{\Set}{\mathbf{Set}}
\newcommand{\Sets}{\ensuremath{\Set}}

\newcommand{\Kel}{\mathbf{Kel}}

\newcommand{\cat}[1]{\mathcal{#1}} 
\newcommand{\escat}[1]{\cat{#1}}

\newcommand{\Psh}[1]{\widehat{#1}}
\newcommand{\Shv}{\mathbf{Sh}}

\newcommand{\calA}{\ensuremath{\escat{A}}} 
\newcommand{\calC}{\ensuremath{\escat{C}}} 
\newcommand{\calE}{\ensuremath{\escat{E}}} 
\newcommand{\calF}{\ensuremath{\escat{F}}} 
\newcommand{\calH}{\ensuremath{\escat{H}}} 
\newcommand{\calJ}{\ensuremath{\escat{J}}} 
\newcommand{\calL}{\ensuremath{\escat{L}}} 
\newcommand{\calS}{\ensuremath{\escat{S}}} 
\newcommand{\calV}{\ensuremath{\escat{V}}} 
\newcommand{\calW}{\ensuremath{\escat{W}}} 
\newcommand{\calX}{\ensuremath{\escat{X}}} 
\newcommand{\calY}{\ensuremath{\escat{Y}}} 

\newcommand{\oalpha}{\overline{\alpha}} 
\newcommand{\obeta}{\overline{\beta}} 
\newcommand{\osigma}{\overline{\sigma}}
\newcommand{\otau}{\overline{\tau}} 
\newcommand{\otheta}{\overline{\theta}}
\newcommand{\okappa}{\overline{\kappa}}

\newcommand{\opCat}[1]{\ensuremath{{#1}^{\mathrm{op}}}}

\newcommand{\Nat}{\ensuremath{\mathbb{N}}}

\newcommand{\twopl}[2]{\ensuremath{\left\langle #1, #2 \right\rangle}}

\newcommand{\interior}{\textnormal{int}}

\newcommand{\enrichedIn}[1]{{#1}\textnormal{-}\mathbf{Cat}}

\newcommand{\wk}[1]{\mathtt{k}{#1}}  

\theoremstyle{plain}
\newtheorem{theorem}{Theorem}[section]
\newtheorem{corollary}[theorem]{Corollary}
\newtheorem{proposition}[theorem]{Proposition}
\newtheorem{lemma}[theorem]{Lemma}

\theoremstyle{definition}
\newtheorem{definition}[theorem]{Definition}
\newtheorem{example}[theorem]{Example}

\title{On the relation between continuous and combinatorial}
\author{F. Marmolejo and M. Menni}

\begin{document}
\maketitle

\begin{abstract}
Axiomatic Cohesion proposes that the contrast between cohesion and non-cohesion may be expressed by means of a geometric morphism ${p : \calE \rightarrow \calS}$ (between toposes) with certain special properties that allow to effectively use the intuition that the objects
of $\calE$ are `spaces' and those of $\calS$ are `sets'. Such geometric morphisms are called {\em (pre-)\linebreak[4]cohesive}. We may also say that $\calE$ is pre-cohesive (over $\calS$).
In this case, the topos $\calE$ determines an $\calS$-enriched `homotopy' category. The purpose of the present paper is to study certain aspects of this  homotopy theory.
We introduce {\em weakly Kan} objects in a pre-cohesive topos, which are analogous to Kan complexes in the topos of simplicial sets.
Also,  given a geometric morphism ${g:\calF \rightarrow \calE}$ between pre-cohesive toposes $\calF$ and $\calE$ (over the same base), we define what it means for $g$ to {\em preserve pieces}. We prove that if $g$ preserves pieces then  it induces an adjunction between the homotopy categories determined by $\calF$ and $\calE$, and that the direct image ${g_*:\calF \rightarrow \calE}$ preserves weakly Kan objects. These and other results support the intuition that the inverse image of $g$ is `geometric realization'. In particular, since Kan complexes are weakly Kan in the pre-cohesive topos of simplicial sets,  the result relating $g$ and weakly Kan objects is analogous to the fact that the singular complex of a space is a Kan complex.
\end{abstract}

\tableofcontents

\section{Introduction}

   Johnstone explains in his 1979 paper \cite{Johnstone79} that one of its aims is to show that there are `convenient categories of spaces' that are toposes. To achieve this aim he constructs the {\em topological topos} (denoted here by $\calJ$) and proves that $\calJ$ embeds the category of sequential spaces. Moreover, he presents the geometric realization functor as the inverse image of a geometric morphism ${\calJ \rightarrow \Psh{\Delta}}$ to the topos of simplicial sets.

   According to  \cite{Johnstone79}, the idea that the realization/singular adjunction resembles a geometric morphism was given a precise form by Joyal who observed that, if we consider the interval ${[0, 1]}$ in $\Set$ as a model of the theory classified by $\Psh{\Delta}$, then we obtain a geometric morphism  ${\Set \rightarrow \Psh{\Delta}}$ whose inverse image produces the underlying set of the usual geometric realization. The desire to `topologize' this observation is mentioned loc.~cit.\ as part of the interest in the topos $\calJ$.
   
   As a related example we can mention Proposition~{10.6} in \cite{Menni2014a}. There, the ingredients are a cohesive topos  ${f:\calF \rightarrow \Set}$  such that ${\calF}$ embeds the monoid of piecewise-linear endos on the interval ${[0, 1]}$, and a geometric morphism ${g:\calF \rightarrow \Psh{\Delta}}$ whose inverse image ${g^*:\Psh{\Delta} \rightarrow \calF}$ sends the total order with two elements ${[1] \in \Delta}$ to the interval ${[0, 1]}$ in $\calF$. It follows that the composite product-preserving left adjoint ${f_* g^* : \Psh{\Delta} \rightarrow \Set}$ sends ${[1] \in \Delta}$ to ${[0, 1]}$ in $\Set$. In other words, the set of points of ${g^* X}$ coincides with the underlying set of the classical geometric realization of the simplicial set $X$.

   From a more general perspective \cite{Lawvere86, Lawvere07}, the geometric morphisms ${\calJ \rightarrow \Psh{\Delta}}$ and ${\calF \rightarrow \Psh{\Delta}}$ are just two examples of  transformations between toposes of spaces. Indeed,  transformations from a topos of `topological' or `piecewise linear' spaces to one of `combinatorial' spaces. 
   
   The appearance of the topos $\Psh{\Delta}$ in both examples is circumstantial. (In fact, the piecewise linear example was devised so as to be directly related to Johnstone's.) In general, we may expect other toposes and, in fact, Lawvere has proposed in \cite{LawvereMail30sep2011} a concrete guide to organize geometric morphisms whose inverse images are `geometric realizations'. The idea is to organize them in terms of their codomains, which are toposes, and so ``some light is shed on their particularity by determining what kind of structure they classify [...]. Concretely, there are many different theories of algebraic structure for which the unit interval is a model, and having chosen one, this structure should be preserved by geometric realization".

  The purpose of the present paper is to axiomatize the concept of a realization/singular adjunction in the context of Axiomatic Cohesion as formulated in \cite{Lawvere07} and pursued in \cite{Johnstone2011, MenniExercise, Menni2014a, LawvereMenni2015}. This will lead us to the definition of a pieces-preserving geometric morphism. In order to argue that the definition is sensible we revisit the classical material as exposed in \cite{GabrielZisman} and examine it from the perspective suggested by Section~{III} in \cite{Lawvere07} and the unpublished \cite{LawvereMail26oct2011, LawvereMail30sep2011}. Indeed, let us quote from \cite{LawvereMail26oct2011} (with some very minor edition):
\begin{quotation}
According to the paradigm set by Milnor, the relation between continuous and combinatorial is a pair of adjoint functors called traditionally singular and realization. (``Singular", as emphasized by Eilenberg, means that the figures on which the combinatorial structure of a space lives should not be required to be monomorphisms,  in order that that structure should be functorial wrt all continuous changes of space; ``realization" refers to a process  analogous  to the passage from blueprints to actual buildings of beton and steel). As emphasized by Gabriel and Zisman, the exactness of realization forces us to refine the default notion of space  itself, in the direction proposed by Hurewicz in the late 40s and well-described by J.~L.~Kelley in 1955. Further refinements suggest that the notion of continuous could well be taken as a topos, of a cohesive (or gros) kind. The exactness of realization is an example of the striving to make the surrogate combinatorial topos (= having a site with finite homs ???)
describe the continuous category as closely as possible. For example the finite products of
combinatorial intervals might be required to admit the diagonal maps that their realizations have. There is one point however where perfect agreement cannot be achieved (Is this a theorem?): the contrast
between continuous and combinatorial forced Whitehead to introduce a specific notion he called weak
equivalence, as explained by Gabriel-Zisman, in order to extract the correct homotopy category. The
contrast can readily be read off of my list of axioms for Cohesion (TAC): the reasonable combinatorial
toposes satisfy all but one of the axioms, but only the continuous examples satisfy it. That Continuity axiom (preservation of infinite products by pizero) was introduced in order to obtain  homotopy types that are ``qualities" in an intuitive sense (as they should be automatically in the continuous case). 
\end{quotation}

   Let us emphasize some of the key guiding ideas:
\begin{enumerate}
\item Axiomatic Cohesion as a general theory of `categories of spaces' emphasizing the ubiquity of toposes of spaces, capable of distinguishing `combinatorial' and continuous examples,  and containing an intrinsic homotopy theory.
\item Realization/singular as a geometric morphism from a continuous topos of spaces to a combinatorial one.
\item Realization as striving to make the surrogate combinatorial topos describe the continuous one as closely as possible.
\item Agreement may not be perfect, as manifested in the classical case of simplicial sets by the need to consider weak equivalences. It is relevant to mention here also the role of Kan complexes.
\item The contrast between continuous and combinatorial can be read off the definition of cohesive topos via the {\em Continuity} axiom. If Continuity holds then the associated homotopy types are ``qualities".
\end{enumerate}

  We have said a few words already about geometric realization as the inverse image of a geometric morphism; so let us concentrate on the last two items. Every pre-cohesive topos $\calE$ over another topos $\calS$ has an associated `homotopy' or {\em Hurewicz} ($\calS$-enriched) category ${\bfh\calE}$. The details will be given later. For the moment it suffices to say that failure of Continuity affects the relation between ${\bfh\calE}$ and $\calS$. Moreover, in the particular case of $\Psh{\Delta}$, the failure of Continuity explains, to certain extent, the relevance of Kan complexes.
  
  In Section~\ref{SecWK} we introduce the notion of weakly Kan object (in a pre-cohesive topos $\calE$). It follows easily from the definition that Continuity holds for $\calE$ if and only if every object in $\calE$ is weakly Kan. So, in a rough sense, the `size' of the subcategory of weakly Kan objects is a measure of the validity of Continuity in $\calE$.
  
  In Section~\ref{SecWKchar} we give a characterization of weakly Kan objects. It follows from this characterization that Kan complexes are weakly Kan objects in the topos of simplicial sets.
It is for this reason that we said above that the failure of Continuity explains, to some extent, the relevance of Kan complexes.

     In Section~\ref{SecPiecePreservingAdjunctions} we define what it means for a pre-cohesive geometric morphism to preserve pieces. In Section~\ref{SecExamplesOfPpGm} we discuss some examples. In particular, we show that some of the typical examples are surjective, as in the case of the geometric morphism ${\calJ \rightarrow \Psh{\Delta}}$ from the topological topos; but notice that we give a very different proof than that in \cite{GabrielZisman}. Our proof uses a new simple sufficient condition, established in Section~\ref{SecSurjective},  for filtering functors to induce surjections. 
     
     In Section~\ref{SecPPandWK} we show that if  ${g:\calF \rightarrow \calE}$ (over the same base) preserves pieces then the direct image ${g_*:\calF \rightarrow \calE}$ preserves weakly Kan objects. This is analogous to the fact that the singular complex of a space is a Kan complex.
     
     In Section~\ref{SecHomotopy} we discuss in more detail the passage to homotopy. We explain in some detail the construction of the Hurewicz category associated to a pre-cohesive topos and prove that every pieces-preserving ${g:\calF \rightarrow \calE}$ (over $\calS$) induces an $\calS$-enriched adjunction between the Hurewicz categories ${\bfh\calF}$ and ${\bfh\calE}$.

   So, in a sense, this is a paper about the foundations of Homotopy Theory. In this sense then, some readers may find surprising the lack of references to Quillen's work \cite{QuillenHA} or other related approaches  such as \cite{KanEtAlBook}. The reason is that the homotopy category associated to a pre-cohesive topos is not constructed by inverting some class of arrows. Needless to say, it is expected that at some point a comparison between the approaches will be done.

   We will recall some of the main definitions and examples in Section~\ref{SecPreCohesive} but the reader will be assumed to be familiar with \cite{Lawvere07} and \cite{Menni2014a}. Further details and examples may be found in \cite{Johnstone2011, MenniExercise}.

\section{Pre-cohesive toposes}
\label{SecPreCohesive}

   Let $\calE$ and $\calS$ be toposes.
   A geometric morphism ${p:\calE \rightarrow \calS}$  is called {\em pre-cohesive} if the adjunction ${p^* \dashv p_*}$ extends to a string of adjoint functors ${p_! \dashv p^* \dashv p_* \dashv p^!}$ such that ${p^* , p^!:\calS\rightarrow \calE}$ are fully faithful, the canonical natural transformation ${\theta:p_* \rightarrow p_!}$ is epi ({\em Nullstellensatz}) and the leftmost adjoint ${p_!:\calE \rightarrow \calS}$ preserves finite products. 
   
   The conditions defining pre-cohesive morphisms allow to effectively use the intuition that the objects of $\calE$ are `spaces', those of $\calS$ are `sets', that ${p_* X}$ is the set of `points' of the space $X$, and that ${p_! X}$ is the set of `pieces' or `connected components'.
  So, for example, a space $X$ in $\calE$ is said to be {\em connected} if ${p_! X = 1}$.
   
  Intuition should be taken seriously. In particular, notice that any equivalence ${p:\calE \rightarrow \calS}$ is pre-cohesive. Roughly speaking, the definition of pre-cohesive geometric morphism allows examples where the concepts of `point' and `piece' coincide.
   More precisely, a pre-cohesive ${p:\calE\rightarrow \calS}$ is called a {\em quality type} if the canonical ${\theta:p_* \rightarrow p_!}$ is an iso. On the other hand, a pre-cohesive ${p:\calE \rightarrow \calS}$ is called {\em sufficiently cohesive} if ${p_! \Omega = 1}$ (that is, the subobject classifier of $\calE$ is connected). Intuitively, sufficiently cohesive examples are those where points and pieces are different concepts. This can be made precise as in Proposition~3 of  \cite{Lawvere07}: if the pre-cohesive ${p:\calE \rightarrow \calS}$ is both sufficiently cohesive and a quality type then $\calS$ is inconsistent.
   In the case of presheaf toposes, this contrast may be strengthened to a dichotomy. We explain this in more detail in the next paragraph.
   
    Let $\calC$ be a small category whose idempotents split and let ${p:\Psh{\calC}\rightarrow \Set}$  be the associated presheaf topos over $\Set$. With different terminology, it is proved in \cite{Johnstone2011} that $p$ is pre-cohesive if and only if $\calC$ has terminal object and every object of $\calC$ has a point (i.e.\ a map from the terminal to that object). Corollary~{2.11} in \cite{MenniExercise} shows that in this case, $p$ is sufficiently cohesive if and only if some object of $\calC$ has two distinct points. Proposition~{4.5} in \cite{Menni2014a} shows that the pre-cohesive $p$ is a quality type if and only if the terminal object of $\calC$ is also initial. Summarizing, a pre-cohesive presheaf topos is either sufficiently cohesive or a quality type. We now recall some examples.
    
    The simplest example of a small category with terminal object and two distinct points is ${\Delta_1}$. It follows that the pre-cohesive topos ${\Psh{\Delta_1}}$ of reflexive graphs is sufficiently cohesive \cite{Lawvere86}. Similarly for simplicial sets: the pre-cohesive ${\Psh{\Delta} \rightarrow \Sets}$ is also sufficiently cohesive.

\begin{example}[The classifier of striclty bipointed objects]\label{ExStrictlyBipointed}
    The theory may be presented by two constants $0$, $1$ and the sequent
$0=1\vdash \bot$. The corresponding classifier may be described as $\con^\calA$ where $\calA$ is the category
of strictly bipointed finite sets and functions between them that preserve the distinguished points. Notice that $\calA$
is the category of free `bipointed sets' generated by a finite set. Let $I$ in $\calA$ be the free bipointed set on one generator.
The standard theory of classifying toposes implies that $\calA(I,-)$ is the generic strictly bipointed set. (One can prove this
directly or use Proposition~{D3.1.10} in \cite{elephant}.) The characterization above implies that ${\Set^{\calA} = \Psh{\opCat{\calA}} \rightarrow \Set}$ is  pre-cohesive and sufficiently cohesive. Notice that ${\opCat{\calA}}$ has finite products and every object is a finite power of $I$.
\end{example}

   For illustration we will mainly use $\Psh{\Delta}$ and ${\con^\calA}$ but we stress that, as already suggested by the quotation in the introduction, any other topos classifying a theory for which the unit interval is a model may be used for similar purposes.
      Further examples may found in \cite{Lawvere07, Johnstone2011, MenniExercise} and, in particular, Examples~{4.2} and {4.3} in \cite{Menni2014a}.

  We will frequently use the following comparison transformation induced by a finite product preserving functor ${F:\calV \rightarrow \calW}$, between cartesian closed categories $\calV$ and $\calW$.
   For every $U$ and $V$ in $\calV$, the exponential transpose of 
$$\xymatrix{
F(V^U) \times F U \ar[r]^-{\cong} & F(V^U \times U) \ar[r]^-{F \ev} & F V
}$$ 
will be denoted by ${\kappa_{U,V}=:F(V^U) \rightarrow (F V)^{(F U)}}$ or by $\kappa^F$ if we want to
make explicit the functor to which it is applied. 
   Observe that for another finite-product preserving functor ${G:\calW \rightarrow \calX}$, the diagram
$$\xymatrix{
G(F(V^U)) \ar[rrd]_-{\kappa^{G F}_{U, V}}\ar[rr]^-{G(\kappa^F_{U, V})} && G((F V)^{(F U)}) \ar[d]^-{\kappa^G_{F U, F V}} \\
 && (G (F V))^{(G(F U))}
}$$ 
commutes.
   
Let ${p:\calE \rightarrow \calS}$ be a pre-cohesive geometric morphism. Denote the unit and counit of ${p_! \dashv p^*}$ by ${\sigma}$ and $\tau$ respectively. Intuitively, the unit ${\sigma_X:X \rightarrow p^*(p_! X)}$ sends each figure of $X$ to the piece where it lies. On the other hand, ${\tau_X:p_! (p^* A) \rightarrow A}$ is an iso; another expression of the requirement that $p^*$ embeds $\calS$ into $\calE$ as discrete spaces.
   
Since ${p_!:\calE\rightarrow \calS}$ preserves finite products we obtain a natural 
${\kappa_{X, Y}:p_!(Y^X) \rightarrow (p_! Y)^{p_! X}}$. 
%
%
%
%
   The following simple fact will play a relevant role.

\begin{lemma}\label{LemDefWeaklyKanNew} For any $X$ in $\calE$ and $A$ in $\calS$ the following composite
$$\xymatrix{
p_! (X^{p^* A}) \ar[r]^-{\kappa} & (p_! X)^{p_!(p^* A)} \ar[rr]^-{(p_! X)^{\tau^{-1}}} && (p_! X)^A
}$$
is an iso if and only if ${\kappa = \kappa^{p_!}_{p^* A, X}:p_! (X^{p^* A}) \rightarrow (p_! X)^{p_!(p^* A)}}$ is an iso.
\end{lemma}
%

   The pre-cohesive $p$ is said to satisfy the {\em Continuity} condition if for every $X$ in $\calE$ and $A$ in $\calS$, the equivalent conditions of Lemma~\ref{LemDefWeaklyKanNew} hold.
   A geometric morphism ${\calE \rightarrow \calS}$ is called {\em cohesive} if it is pre-cohesive and satisfies the Continuity condition.
   
   The first example of a cohesive and sufficiently cohesive topos  was described in \cite{Lawvere07}. In this case, the base $\calS$ is the topos of finite sets so Continuity reduces to preservation of finite products. The examples of cohesive (and sufficiently cohesive) Grothendieck topos that we are aware of are those described in \cite{Menni2014a}. For definiteness take the site
$(\bflp,K)$. Recall that the
category $\bflp$ has as objects the closed intervals $[a,b]$ with $a\leq b$, $a,b\in\mathbb{R}$, and as morphisms piecewise
linear (continuous) maps: a continuous map $f\pts[a,b]\to[c,d]$ is piecewise linear if $a=b$ or $a<b$ and there is a partition 
$a=r_0<r_1<\cdots<r_{m+1}=b$ of $[a,b]$ such that for every $i=0,\dots, m$, 
the restriction $f|_{[r_i,r_{i+1}]}\pts [r_i,r_{i+1}]\to [c,d]$ is a linear function.
Recall as well that the Grothendieck topology on $\bflp$ is given by means of a basis $K$: $K[a,a]$ consists
only of the identity morphism for any $a\in\mathbb{R}$, whereas for $a<b$, $K[a,b]$ consists of those families of the
form 
\[
\{[r_i,r_{i+1}]
\xy
(0,0)*+{}="1",
(10,0)*+{}="2",
\POS "1" \ar@{^(->} "2",
\endxy
[a,b]\,|\,i=0,\cdots, m\}
\]
where $a=r_0<r_1<\cdots<r_{m+1}=b$ is a partition of $[a,b]$. It is proved in Section~{10} of \cite{Menni2014a}  that the canonical ${\Shv(\bflp,K) \rightarrow \Set}$ is cohesive and sufficiently cohesive. It is also shown loc.\ cit.\ (Lemma~{10.5}) 
that the site $(\bflp,K)$ is subcanonical.

    The following concept is perhaps somewhat ad-hoc but it provides a more concrete grasp of the functor ${p_!}$ and proved to be quite useful in \cite{Menni2014a}.
    
A {\em connector} for  ${p:\calE\rightarrow\calS}$ is a bipointed object ${0, 1: 1\to I}$ in $\calE$ such that  the following diagram
\[
\xy
(0,0)*+{p_*(X^I)}="1",
(25,0)*+{p_* X}="2",
(45,0)*+{p_! X}="3",
\POS "1" \ar@<2.5pt>^{p_* \ev_0} "2",
\POS "1" \ar@<-2.5pt>_{p_* \ev_1} "2",
\POS "2" \ar^{\theta_X} "3",
\endxy
\]
is a coequalizer in $\calS$ for each $X$ in $\calE$. (See Definition~{8.1} in \cite{Menni2014a}.)

For  example, consider the topos of simplicial sets $\widehat{\Delta}$. The representable ${I = \Delta(-,[1])}$ has exactly two points, and Example~{8.10} in \cite{Menni2014a}  shows that they form a connector for ${p\pts\widehat{\Delta}\to\con}$. Example~{8.12} loc.\ cit.\ shows that the classifier of strictly bipointed objects has a connector, and Lemma~{8.13} implies that ${\Shv(\bflp,K) \rightarrow \Set}$ does.

    Intuitively, one pictures a connector as a connected object and, in fact, the connectors in our examples are all connected (i.e. ${p_! I = 1}$); but we stress that we have no use for this condition in our proofs, except for that of Proposition~\ref{proposition1.5}.

\section{Filtering functors inducing surjections}
\label{SecSurjective}

   Gabriel and Zisman prove in Section~{III.3.6} of \cite{GabrielZisman}  that geometric realization ${\Psh{\Delta} \rightarrow \Kel}$ reflects isos, where $\Kel$ is the category of Kelley spaces. In \cite{Johnstone79}, Johnstone cites this result as a proof  that the geometric morphism ${\calJ \rightarrow \Psh{\Delta}}$ from the topological topos $\calJ$ to simplicial sets, whose inverse image sends ${[1] \in \Delta}$ to the unit interval in ${\calJ}$, is surjective. (Recall that a geometric morphism ${g:\calF \rightarrow\calE}$ is surjective if and only if ${g^*:\calE\rightarrow \calF}$ reflects isos.)
   
   So, for pre-cohesive geometric morphisms ${f:\calF \rightarrow \calS}$ and ${p:\calE\rightarrow \calS}$ over the same base, we are led to  entertain the idea of a surjective  geometric morphism ${\calF \rightarrow \calE}$ over ${\calS}$ satisfying some condition(s) typical of a realization/singular adjunction. Moreover, the examples suggest that $f$ may be cohesive and that $p$ may be `combinatorial' in some sense.
   We will exhibit examples in Section~\ref{SecExamplesOfPpGm}, but in this section we concentrate on the issue of surjectivity/reflection-of-isos. We found the proof in \cite{GabrielZisman} somewhat difficult to follow so we decided to give an alternative one that is more in tune with our topos theoretic context.

   Let $\calC$ be a small category and let ${A:\calC \rightarrow \Set}$ be a filtering functor (Definition~{VII.6.2} in \cite{maclane2}).
   Denote the induced geometric morphism by ${f:\Set \rightarrow \Psh{\calC}}$. Recall that for any $P$ in $\Psh{\calC}$, ${f^* P}$ in $\Set$ may be described as the quotient of ${\sum_{C \in \calC} (P C) \times (A C)}$ by the equivalence relation that relates the pairs ${(x, a) \in (P C) \times (A C)}$ and ${(x', a') \in (P C') \times (A C')}$ if and only if  there exists a span 
$$\xymatrix{
C & \ar[l]_-{u} E \ar[r]^{u'} & C'
}$$
in $\calC$ and ${c \in A E}$ such that ${u \cdot c = a}$, ${u' \cdot c = a'}$ and ${x\cdot u = x'\cdot u'}$. 
See discussion following Theorem~{VII.6.3} in \cite{maclane2}.

\begin{definition}\label{LemInterior}
For any $C$ in $\calC$, the {\em interior} of ${A C}$ is the subset ${\interior(A C)  \rightarrow A C}$ given by those ${x \in A C}$ such that: for every  ${u:C' \rightarrow C}$ in $\calC$, if ${(A u) y = u\cdot y = x}$ for some ${y \in A C'}$ then $u$ is split epi.
\end{definition} 
   
   We can now prove a sufficient condition for a filtering functor to induce a surjective geometric morphism.
   
\begin{theorem}\label{ThmCharQuico}
Let ${A:\calC \rightarrow \Set}$ be a filtering functor.
If  ${\interior(A C) \not= \emptyset}$ for every $C$ in $\calC$, then  the geometric morphism ${\Set \rightarrow \Psh{\calC}}$ induced by $A$ is surjective.
\end{theorem}
\begin{proof}
Let ${f:\Set \rightarrow \Psh{\calC}}$ be the geometric morphism induced by $A$. It is enough to prove that for every mono ${\varphi:Q \rightarrow P}$ in $\Psh{\calC}$, if ${f^* \varphi:f^* Q \rightarrow f^* P}$ is iso then so is $\varphi$. In turn, it is enough to prove that ${f^*}$ reflects epis. So assume that ${f^* \varphi:f^* Q \rightarrow f^* P}$ is epi. We need to check that ${\varphi = \varphi_C:Q C \rightarrow P C}$ is surjective for each $C$ in $\calC$. To do this let ${x\in P C}$. By hypothesis, ${\interior(A C) \not= \emptyset}$, so let ${a\in \interior(A C)}$ and consider ${x \otimes a \in f^* P}$. Since ${f^* \varphi}$ is epi, 
\[ x\otimes a = (f^* \varphi)(y \otimes b) = (\varphi y)\otimes b \] 
for some ${y \in Q D}$ and ${b \in A D}$. So there exists a span
$$\xymatrix{
C & \ar[l]_-{u} E \ar[r]^{v} & D
}$$
and ${c \in A E}$ such that ${u \cdot c = a}$, ${v \cdot c = b}$ and ${x \cdot u = (\varphi y)\cdot v}$. Since ${a\in \interior(A C)}$, $u$ has a section ${s:C \rightarrow E}$. Then 
\[ x = x\cdot (u s) = (x\cdot u) \cdot s = ((\varphi y)\cdot v) \cdot s = \varphi(y\cdot (v s))
\]
showing that $x$ is in the image of ${\varphi_C:Q C \rightarrow P C}$.
\end{proof}

The intuition is that ${A:\calC \rightarrow \Set}$ sends each $C$ in $\calC$ to the underlying set of a solid object. Now,
for an arrow $f\pts C'\rightarrow C$ in $\calC$, the image of $A(f)$ intersects the interior of $A(C)$ only if the domain is 
at least of the same ``dimension'' than the codomain. The non-empty interior condition captures the idea that all the
lower dimensional figures fail to cover ${A C}$. This intuition might be more easily grasped if we add a small condition on $\calC$.

\begin{lemma}\label{LemOriginalInt}
For ${A:\calC \rightarrow \Set}$ as above. If every map in $\calC$ factors as a split-epi followed by a mono then, for every $C$ in $\calC$, ${\interior(A C)  \rightarrow A C}$ consists of those ${x \in A C}$ such that: for every  mono ${m:C' \rightarrow C}$ in $\calC$, if ${(A u) y = u\cdot y = x}$ for some ${y \in A C'}$ then $m$ is an iso.
\end{lemma}
\begin{proof}
Easy.
\end{proof}

In this case the intuition is more direct,  ${A:\calC \rightarrow \Set}$ sends each non-iso mono ${C' \rightarrow  C}$ to the inclusion of a `lower dimensional' figure. The
non-empty interior condition captures the idea that all these lower dimensional subfigures
fail to cover ${A C}$.

For instance, consider the usual functor ${A:\Delta \rightarrow \Set}$ that sends ${[n] \in \Delta}$ to the 
simplex ${\Delta^n}$. Fix some ${n \geq 1}$ and consider ${A [n] = \Delta^n}$. Any non-iso mono $[m]\to [n]$
in $\Delta$ factors through some mono $[n-1]\to [n]$ and it is well-known that the induced $A[n-1]\to A[n]$ lies
in a face of $A[n]$ (see, for example, Section~{VII.5} in \cite{maclane}), so the interior of ${A[n]}$ is nonempty.

   More precisely,  recall that one description (see \cite{maclane2},
VIII.7 for instance) has that $\Delta^0$ is a 
point, and for $n\geq 1$, $\Delta^n=\{(t_1,\dots,t_n)\in [0,1]^n | 0\leq t_1\leq \cdots\leq t_n\leq 1\}$. Observe that $\Delta^{\partial_k^i}(t_1,\dots,t_k) = (t_1,\dots,t_{k-1},t_{k},t_k,
t_{k+1},\cdots t_n)$ ($\Delta^{\partial_k^0}$ adds a leftmost zero, and $\Delta^{\partial_k^{k+1}}$ adds a rightmost 1).
From this it is not hard to see that the interior of ${\Delta[n]}$ is ${\{(t_1,\dots,t_n) | 0<t_1<\cdots <t_n<1\}}$, as it is to
be expected. (Of course, the interior of $\Delta^0$ is $\Delta^0$ itself.) We can now apply Theorem \ref{ThmCharQuico}.
 
\begin{corollary}\label{CorFolklore} The geometric morphism ${\Set \rightarrow \Psh{\Delta}}$ whose inverse image sends ${[1] \in \Delta}$ to ${[0, 1]}$ in $\Set$ is surjective.
\end{corollary}

  As a further corollary we obtain an alternative proof of the `topologized' version.
  
\begin{corollary}[Gabriel-Zisman/Johnstone]\label{CorJohnstone} The geometric morphism ${\calJ \rightarrow \Psh{\Delta}}$ whose inverse image sends ${[1] \in \Delta}$ to ${[0, 1]}$ in $\calJ$ is surjective.
\end{corollary}
\begin{proof}
Let us denote the canonical geometric morphism from the topological topos by ${f:\calJ \rightarrow \Set}$ and the realization/singular morphism by ${r:\calJ \rightarrow \Psh{\Delta}}$. The morphism $f$  is local (in the sense that it has a fully faithful right adjoint), so the `points' functor ${f_*:\calJ \rightarrow \Set}$ is the inverse image of a geometric morphism ${c:\Set \rightarrow \calJ}$. The composite
$$\xymatrix{\Set \ar[r]^-{c} & \calJ \ar[r]^-r & \Psh{\Delta}}$$
is such that its inverse image sends ${\Delta(\_, [1]) \in \Psh{\Delta}}$ to ${[0, 1]}$ in $\Set$ so it must coincide with the geometric morphism of Corollary~\ref{CorFolklore}. Hence, the composite ${c^* r^* = f_* r^*:\Psh{\Delta} \rightarrow \Set}$ is faihful. Then ${r^*:\Psh{\Delta} \rightarrow \calJ}$ is faithful, which means that $r$ is surjective as a geometric morphism.
\end{proof}

   It seems relevant to compare the proof of Corollary~\ref{CorJohnstone} with Gabriel and Zisman's proof of conservativity of geometric realization ${\mid\_\mid: \Psh{\Delta} \rightarrow \Kel}$ in III.3.6 of \cite{GabrielZisman}. The relevant part is at the end of page~{53} loc.~cit., where it is proved that if ${f:X \rightarrow Y}$ is a non invertible mono in ${\Psh{\Delta}}$ then its geometric realization ${\mid f \mid}$ is not invertible either. 
   The key auxiliary  fact is that the group of automorphisms of ${\mid\_\mid}$ acts on $Y$ and that the orbits of this action are indexed by the non-degenerate simplices of $Y$ (see Proposition~{III.1.6} loc.~cit.). In turn, the proof of this fact rests on the identification of the group of automorphisms of ${\mid\_\mid: \Psh{\Delta} \rightarrow \Kel}$  with the group of increasing continuous maps ${s: [0, 1] \rightarrow [0, 1]}$ such that ${s 0 = 0}$ and ${s 1 = 1}$ (II.1.3). From a topos theoretic perspective, the last fact is a corollary of the classifying role of ${\Psh{\Delta}}$; while the former is an instance of the more general observation that `algebraic structure is adjoint to semantics' (Theorem~{III.1.2} in \cite{LawvereThesis}).
   As a final remark on this example we stress that the group of automorphisms of the realization functor plays no role in our proof.

For another application of Theorem \ref{ThmCharQuico} consider the classifying topos for strictly bipointed objects from 
Example~\ref{ExStrictlyBipointed}. Let $\calA$ be the category of finite, strictly bipointed sets. It is easy to see that monos and epis in $\calA$ are split and that every map factors as an epi followed by a mono. Thus $\calA\op$ satisfies the condition
in Lema \ref{LemOriginalInt}.
To consider the corresponding filtering functor $\calA\op\to\con$ and the condition on interiors we switch to
the following ``geometric'' version $\bfc$ of $\calA\op$. The objects of $\bfc$ are
cubes, that is, they are of the form $[0,1]^S$ with $S$ a finite set; a morphism $f\pts [0,1]^S\to [0,1]^T$ in $\bfc$ is built up
from projections and the constants 0 and 1 only, that is to say, $f$ is of the form $f=\langle f_t\rangle_{t\in T}$ 
where for every $t\in T$, $f_t\pts[0,1]^S\to [0,1]$ is a projection $\pi_s\pts[0,1]^S\to[0,1]$, with $s\in S$, or the constant zero 
$\ulcorner 0\urcorner\pts[0,1]^S\to[0,1]$, or the constant one $\ulcorner 1\urcorner\pts[0,1]^S \to[0,1]$. Then $\bfc$ and $\calA\op$
are isomorphic. The filtering functor in question is the inclusion $\bfc\to\con$. Now, intuition indicates that the interior of $[0,1]^S$
are those points that avoid the border of the cube (no coordinate equals zero and no coordinate equals one) an also avoid the
``diagonals'' (no two coordinates are equal). Indeed, observe that $f=\langle f_t\rangle_{t\in T}
\pts [0,1]^S\to[0,1]^T$ is mono in $\bfc$ iff for every $s\in S$
exists $t$ such that $f_t=\pi_s$. Now, if $f$ is not an iso, then there is an $s$ with two such $t$'s, or there is a $t\in T$ such
that $f_t$ is constant zero or constant 1. It is clear then that any point in the image of $A(f)$ has a zero coordinate or has a
coordinate that is one, or has two equal coordinates. Thus any point $(x_1,\dots,x_n)\in A([0,1]^S)=[0,1]^S$ 
such that for all $i$, $x_i\not=0,1$ and for every $i\not=j$, $x_i\not= x_j$, is in the interior. We may conclude  the following.

\begin{corollary}\label{bipointedsurjection} 
Let $\calA$ be the category of finite, strictly bipointed sets, and let $I$ the free bipointed object in one generator in $\calA$.
The geometric morphism ${\Set \rightarrow \con^\calA}$ whose inverse image sends $\calA(I,-)$ to ${[0, 1]}$ in
$\Set$ is surjective.
\end{corollary}

\section{Weakly Kan objects}
\label{SecWK}

   Let ${p:\calE \rightarrow \calS}$ be a pre-cohesive geometric morphism.

\begin{definition}\label{DefWeaklyKan}
An object $X$ in $\calE$ will be called {\em weakly Kan} if, for every $A$ in $\calS$, the equivalent conditions of Lemma~\ref{LemDefWeaklyKanNew} hold. That is, the following composite
$$\xymatrix{
p_! (X^{p^* A}) \ar[r]^-{\kappa} & (p_! X)^{p_!(p^* A)} \ar[rr]^-{(p_! X)^{\tau^{-1}}} && (p_! X)^A
}$$
is an iso or, equivalently,  ${\kappa = \kappa^{p_!}_{p^* A, X}:p_! (X^{p^* A}) \rightarrow (p_! X)^{p_!(p^* A)}}$ is an iso.
\end{definition}

   Roughly speaking, $X$ is weakly Kan if $p_!$ preserves arbitrary powers of $X$. 
   Notice that $p$ satisfies the Continuity condition if and only if  every object in $\calE$ is weakly Kan. 
   
   The terminology intends to invite the reader to think of a weakly Kan object as something like a Kan complex. We will show that, in the case of simplicial sets, every Kan complex in the usual sense is weakly Kan as an object in $\Psh{\Delta}$.

   Let ${\wk{\calE} \rightarrow \calE}$ be the full subcategory determined by the weakly Kan objects. 

\begin{lemma}\label{LemDiscreteImpliesWK} For every object $B$ in $\calS$, ${p^* B}$ in $\calE$ is weakly Kan. 
\end{lemma}
\begin{proof}
It is not hard to see that 
$$\xymatrix{
(p_! (p^* B))^{p_! (p^* A)} \ar[rr]^-{(\tau_B)^{\tau_A^{-1}}} && B^A \ar[r]^-{\tau^{-1}} & 
p_!(p^* (B^A)) \ar[r]^-{p_! \kappa} & p_!((p^* B)^{p^*A})
}$$
is the inverse of ${\kappa : p_!((p^* B)^{p^*A})\rightarrow  (p_! (p^* B))^{p_! (p^* A)}}$. Indeed, we know that the morphism
${\kappa\pts p^*(B^A)\to (p^*B)^{p^*A}}$ 
is an iso (Corollary A1.5.9 in \cite{elephant}), so it suffices to show that one of the composites is the identity.
\end{proof}

  In other words, discrete objects are weakly Kan.
   In particular, the category ${\wk\calE}$ has initial and terminal object and the inclusion ${\wk\calE \rightarrow \calE}$ preserves them.
   
\begin{lemma}
The subcategory ${\wk\calE \rightarrow \calE}$ is closed under finite products.
\end{lemma}
\begin{proof}
We know that the terminal object is weakly Kan. 
For objects ${X, Y\in \wk\calE}$ and $A$ in $\calS$ it is not hard to see that the canonical
morphism ${\kappa\pts p_!((X\times Y)^{p^*A})\to p_!(X\times Y)^{p_! (p^* A)}}$ is the composite
$$\xymatrix{
p_!(X^{p^*A}\times Y^{p^*A}) \ar[r]^-{\cong} & p_!(X^{p^*A})\times p_!(Y^{p^*A}) \ar[r]^-{\kappa\times\kappa} & (p_! X)^{p_! (p^* A)}\times (p_! Y)^{p_! (p^* A)} \ar[d]^-{\cong} \\
p_!((X\times Y)^{p^*A}) \ar[u]^-{\cong} \ar[r]_-{\kappa} & (p_!(X\times Y))^{p_! (p^* A)} & \ar[l]^-{\cong}  (p_!X\times p_!Y)^{p_! (p^* A)} \\
}$$
where the isomorphisms come from $p_!$, ${(\_)^{p^*A}}$ and ${(\_)^{p_! (p^* A)}}$ preserving finite products.
\end{proof}
   
   The subcategory ${\wk\calE \rightarrow \calE}$ is also closed under arbitrary powers in the following sense.

\begin{lemma}\label{LemPowersOfWK}
For every $X$ in ${\wk\calE}$ and $A$ in $\calS$, ${X^{p^* A}}$ is also in ${\wk\calE}$.
\end{lemma}
\begin{proof}
For every ${B\in \calS}$
\[ (X^{p^* A})^{p^* B} \cong X^{(p^* A) \times (p^* B)} \cong  X^{p^*(A\times B)} \]
so ${\kappa:p_!((X^{p^* A})^{p^* B}) \rightarrow (p_!(X^{p^* A}))^{p_!(p^* B)}}$ is as iso as ${\kappa:p_!(X^{p^*(A\times B)}) \rightarrow (p_! X)^{p_!(p^*(A\times B))}}$.
\end{proof}

   At this point one may wonder if weakly Kan is equivalent to discrete. We will show that in the more interesting cases,  the subcategory ${\wk\calE \rightarrow \calE}$ is much bigger. In fact, we will show that if $\calE$ is covered by a cohesive topos $\calF$ then 
   ${\wk\calE \rightarrow \calE}$ contains all the  `singular complexes' of the objects in $\calF$. 

\section{Weakly Kan objects in the presence of a connector}
\label{SecWKchar}

   In this section we give a characterization of weakly Kan objects in pre-cohesive toposes over $\Set$ equipped with a connector. In order to do this we first need to study the following concept.
   
\begin{definition}\label{DefPowerfulCoequalizer} In a cartesian closed category, a fork as on the left below
$$\xymatrix{
E \ar[r]<+.5ex>^-{e_0} \ar[r]<-.5ex>_-{e_1} & D \ar[r]^-q & Q &&
 E^A \ar[r]<+.5ex>^-{e_0^A} \ar[r]<-.5ex>_-{e_1^A} & D^A \ar[r]^-{q^A} & Q^A
}$$
is {\em powerful} if for every object $A$, the fork on the right above is a coequalizer.
\end{definition}

   Notice that a powerful fork as above is always a coequalizer. Indeed, this follows from the case ${A = 1}$.

   Assume now that ${p:\calE \rightarrow \calS}$ is a pre-cohesive geometric morphism equipped with a connector  ${0, 1:1 \rightarrow I}$. 
   
\begin{lemma}\label{LemCharWKwithConnectors}
An object $X$ in $\calE$ is weakly Kan if and only if  the fork below
$$\xymatrix{
p_* (X^I) \ar[rr]<+.5ex>^-{p_* \ev_0} \ar[rr]<-.5ex>_-{p_* \ev_1} && p_* X \ar[r]^-{\theta} & p_! X \\
}$$
is a powerful coequalizer in $\calS$.
\end{lemma}
\begin{proof}
Let $X$ in $\calE$ and $A$ in $\calS$.
Since $I$ is a connector, the two forks below 
$$\xymatrix{
p_*(X^I) \ar[rr]<+.5ex>^-{p_* \ev_0} \ar[rr]<-.5ex>_-{p_* \ev_1} && p_*X \ar[r]^-{\theta} & p_! X 
& p_*((X^{p^* A})^I) \ar[rr]<+.5ex>^-{p_* \ev_0} \ar[rr]<-.5ex>_-{p_* \ev_1} && p_*(X^{p^* A}) \ar[r]^-{\theta} & p_! (X^{p^* A}) 
}$$
are coequalizers. The diagram below
$$\xymatrix{
p_*((X^{p^* A})^I) \ar[d]_-{\cong} \ar[rr]<+.5ex>^-{p_* \ev_0} \ar[rr]<-.5ex>_-{p_* \ev_1} && p_*(X^{p^* A}) \ar[dd]^-{\cong} \ar[r]^-{\theta} & p_! (X^{p^* A}) \ar[dd] \\
p_*((X^I)^{p^* A}) \ar[d]_-{\cong} \\
(p_* (X^I))^A  \ar[rr]<+.5ex>^-{(p_* \ev_0)^A} \ar[rr]<-.5ex>_-{(p_* \ev_1)^A} && (p_* X)^A  \ar[r]^-{\theta^A} & (p_! X)^A \\
}$$
commutes in the evident sense and the left and middle vertical maps are iso.
Therefore, the bottom fork is a coequalizer   if and only if the right vertical map is an iso.
\end{proof}

   We now concentrate on powerful coequalizers in the topos $\Set$ of sets and functions.
Fix a coequalizer ${q e_0 = q e_1}$ as above.    Let us write ${\rightsquigarrow}$ for the relation on ${D}$ given by  the image of ${\twopl{e_0}{e_1}:E \rightarrow D\times D}$.  Write ${\sim}$ for the reflexive and symmetric closure of $\rightsquigarrow$.
   
   Let ${\Nat_\infty}$ be the usual poset of natural numbers extended with terminal object denoted by $\infty$. The {\em distance} from $x$ to $y$ in ${p_* X}$ is the least ${n \in \Nat_{\infty}}$ such that there are ${x_1, \ldots, x_n}$ such that ${x = x_0 \sim x_1 \sim x_2 \sim \ldots \sim x_{n-1} \sim x_n = y}$. The distance from $x$ to $y$ will be denoted by ${d(x, y)}$. For example, ${d(x, x) = 0}$. Also, ${q x = q y}$ if and only if ${d(x, y) < \infty}$.

\begin{lemma}\label{LemCharPowerfulCoeq}
For any coequalizer in $\Set$ as above the following are equivalent:
\begin{enumerate}
\item The coequalizer is powerful.
\item The next fork is a coequalizer
$$\xymatrix{
 E^\Nat \ar[r]<+.5ex>^-{e_0^\Nat} \ar[r]<-.5ex>_-{e_1^\Nat} & D^\Nat \ar[r]^-{q^\Nat} & Q^\Nat
}$$
\item (finite distances are bounded)  There exists an ${n \in \Nat}$ such that for every ${x, y \in D}$, ${q x = q y}$ implies ${d(x, y) \leq n}$.
\end{enumerate}
\end{lemma}
\begin{proof}
The first item  trivially implies the second. To prove that the second implies the third assume, for the sake of contradiction,  that distances are not bounded. Then, for every ${m \in \Nat}$ there are ${x_m, y_m \in D}$ such that ${d(x_m, y_m) \geq m}$. Then the indexed families ${\vec{x} = (x_m \mid m\in \Nat)}$ and ${\vec{y} = (y_m \mid m\in \Nat)}$ in ${D^{\Nat}}$ are such that ${q^{\Nat}\vec{x} = q^{\Nat} \vec{y}}$. As the fork in the second item is a coequalizer, ${d(\vec{x}, \vec{y}) < \infty}$. Then, for every ${m \in \Nat}$, ${d(x_m, y_m) \leq d(\vec{x}, \vec{y})}$. Absurd.
Finally, assume that the third item holds. We need to show that the following fork
$$\xymatrix{
 E^A \ar[r]<+.5ex>^-{e_0^A} \ar[r]<-.5ex>_-{e_1^A} & D^A \ar[r]^-{q^A} & Q^A
}$$
is a coequalizer, where $A$ is some set.
Let ${\vec{x}, \vec{y}}$ in ${D^A}$ be such that ${q^A \vec{x} = q^A \vec{y}}$ in ${Q^A}$.
Then ${q(\vec{x} a) = q(\vec{y} a)}$ in $Q$ for every ${a\in A}$. By hypothesis, there is an ${n\in \Nat}$ such that ${d(\vec{x} a, \vec{y} a) \leq n}$ for every ${a\in A}$. Hence, there are $
{r_{a, 1}, \ldots, r_{a, n}}$ in $D$ such that ${\vec{x} a \sim r_{a, 1} \sim \ldots r_{a, n} = \vec{y} a}$. For any ${1\leq j \leq n}$, we can consider the family $\vec{r_j}$ such that ${\vec{r_j} a = r_{j, a}}$ and it is easy to see that ${\vec{x} \sim \vec{r_1} \ldots \sim\vec{y}}$. 
\end{proof}

Lemmas~\ref{LemCharWKwithConnectors} and \ref{LemCharPowerfulCoeq} give a fairly concrete characterization of the weakly Kan objects in a pre-cohesive topos ${p:\calE \rightarrow \Set}$ equipped with a connector. For instance, in
the category $\widehat{\Delta_1}$ of reflexive graphs, a graph $G$ is weakly Kan if and only if there exists an $n\in\Nat$
with the property that, if any two nodes $x$ and $y$ belong to the same connected component of $G$, then there is a path
(of back and forth arrows) of length at most $n$ that connects $x$ and $y$. In the category $\widehat{\Delta}$ of simplicial
sets, an object is weakly Kan if and only if its underlying reflexive graph (its 1-skeleton) is weakly Kan in 
$\widehat{\Delta_1}$. This characterization of weakly Kan objects in simplicial sets benefited from discussions with
Luis Turcio.

Lemma~\ref{LemCharWKwithConnectors} also suggests the following interesting sufficient condition 
in a more general context.
   Recall that a fork as below (in a category with pullbacks)
\[
\xy
(0,0)*+{A}="1",
(15,0)*+{B}="2",
(30,0)*+{C}="3",
\POS "1" \ar@<3pt>^{e_0} "2",
\POS "1" \ar@<-3pt>_{e_1} "2",
\POS "2" \ar^{e} "3",
\endxy
\]
is {\em exact} if it is a coequalizer and a pullback. The same fork will be called {\em quasi-exact} if it is a coequalizer and the induced morphism ${A\rightarrow \ker e}$ is regular epi, where ${\ker e}$ is the kernel pair of $e$. 
   Any regular functor (between regular categories) preserves exact forks. See, for example, Section~{A1.3} in \cite{elephant}. It is easy to check that regular functors also preserve quasi-exact forks.

   Let ${p:\calE \rightarrow \Set}$ be a pre-cohesive topos equipped with a connector ${0, 1:1 \rightarrow I}$.
An object $X$ in $\calE$ will be called {\em navigable} if the fork
\[
\xymatrix{
p_* (X^I) \ar[rr]<+.5ex>^-{p_* \ev_0} \ar[rr]<-.5ex>_-{p_* \ev_1} && p_* X \ar[r]^-{\theta} & p_! X \\
}
\]
is quasi-exact. Intuitively, $X$ is navigable if, assuming that you can move from $a$ to $b$ in $X$ and also from $b$ to $c$, then there is also a way to move from $a$ to $c$.

   An object $A$ in a topos $\calS$ is {\em internally projective} if the right adjoint ${\Pi_A:\calS/A \rightarrow \calS}$ is a regular functor or, equivalently, ${(\_)^A:\calE \rightarrow\calE}$ preserves epis. The topos $\calS$ satisfies the {\em internal axiom of choice} (IAC) if every object $A$ is internally projective.

\begin{proposition}\label{PropNavigableImpliesWK}
If $\calS$ satisfies IAC and $X$ in $\calE$ is navigable then $X$ is weakly Kan.
\end{proposition}
\begin{proof}
Follows from Lemma~\ref{LemCharWKwithConnectors} because since $\calS$ satisfies IAC the coequalizer in question is powerful.
\end{proof}

   Consider for example the pre-cohesive topos of simplicial sets equipped the connector described in Section~\ref{SecPreCohesive}.
   
\begin{corollary}
Every Kan complex is weakly Kan for ${p:\Psh{\Delta} \rightarrow \Set}$.
\end{corollary}
\begin{proof}
Follows from Proposition~\ref{PropNavigableImpliesWK} and  \cite{JoyalTierneyQuad47}. Briefly, every Kan complex is navigable.
In more detail, for every simplicial set $X$, ${p_* (X^I) = \Psh{\Delta}(1, X^I) = \Psh{\Delta}(I, X) = X [1]}$. In the notation of \cite{JoyalTierneyQuad47}, ${p_* (X^I) = X_1}$.
Similarly, ${p_* X = X_0}$. So the fact that ${0, 1:1 \rightarrow I}$ is a connector for ${p:\Psh{\Delta}\rightarrow \Set}$ coincides with the fact stressed at the beginning of Section~{3.2} loc.\ cit., namely, that the following diagram
$$\xymatrix{
X_1 \ar@<+.5ex>[r]^-{d^0} \ar@<-.5ex>[r]_-{d^1} & X_0 \ar[r] & p_! X
}$$
is a coequalizer for every $X$ in $\Psh{\Delta}$. Joyal and Tierney denote the induced relation on $X_0$ by $\sim$ and show that if $X$ is a Kan complex then $\sim$ is an equivalence relation. But this simply means that the fork above is quasi-exact. That is, $X$ is navigable.
\end{proof}

\section{Geometric morphisms that preserve pieces}
\label{SecPiecePreservingAdjunctions}

  As suggested in the introduction, Milnor's geometric realization interpreted by Joyal, Johnstone and Lawvere using toposes leads to the idea of a geometric morphism ${g:\calF \rightarrow \calE}$ over a base $\calS$ as in the diagram on the left below
$$\xymatrix{
\calF \ar[rd]_-f \ar[r]^-g & \calE \ar[d]^-p && \calF \ar[rd]_-f \ar[r]^-g & \Psh{\Delta} \ar[d]^-p\\
 & \calS && & \Set
}$$
where $f$ is cohesive and $\calE$ is a `combinatorial' pre-cohesive topos. Proposition~{10.6} in \cite{Menni2014a} shows a concrete example by making explicit a cohesive topos ${f:\calF \rightarrow\Sets}$ of `piecewise linear spaces' and a geometric morphism ${g:\calF \rightarrow \Psh{\Delta}}$ such that the composite of the subtopos ${f_* \dashv f^!:\Set\rightarrow \calF}$ followed by ${g:\calF \rightarrow \Psh{\Delta}}$ is the geometric morphism ${\Set \rightarrow \Psh{\Delta}}$ whose inverse image sends each simplicial set $X$ to the underlying set of Milnor's geometric realization of $X$. The result cited above highlights a natural iso ${\lambda:p_! g_* \rightarrow f_!}$ formalizing the idea that ${g_*:\calF \rightarrow \Psh{\Delta}}$ preserves pieces. We will show that this example is an instance of a more general notion of a pieces-preserving geometric morphism ${g:\calE \rightarrow \calF}$.

   As mentioned above, our main interest is in geometric morphisms, but for many of the calculations only the most basic facts about adjunctions and finite products are needed, so we start with a very basic setting and add hypotheses as they are needed.

\subsection{The preservation of indexed coproducts}

   Let $\calE$ and $\calS$ be categories and let ${p^* \dashv p_*:\calE \rightarrow \calS}$ be an adjunction. 
   Let ${g^* \dashv g_*:\calF \rightarrow \calE}$ be another adjunction (denoted also by ${g:\calF \rightarrow \calE}$) and let ${f:\calF \rightarrow \calS}$ be the composite adjunction, so that ${f_* = p_* g_*:\calF \rightarrow \calS}$ and  ${f^* = g^* p^*:\calS \rightarrow \calE}$.    We picture the situation as follows
$$\xymatrix{
\calF \ar[rd]_-f \ar[r]^-g & \calE \ar[d]^-p \\
 & \calS
}$$
with $f$ `over' $p$, and we devise a notation that explicitly relates the unit and counit of $p$ with that of $f$.

   The unit and counit of $p$ will be denoted by ${\alpha:1_{\calS} \rightarrow p_* p^*}$ and ${\beta:p^* p_* \rightarrow 1_{\calE}}$. `Over' it we denote the unit and counit of $f$ by ${\oalpha:1_{\calS} \rightarrow f_* f^*}$ and ${\obeta:f^* f_* \rightarrow 1_{\calE}}$ respectively. In order to relate these natural transformations we introduce a different name for the unit and counit of $g$.

We denote the unit and counit of $g$ by ${\nu:1_{\calE} \rightarrow g_* g^*}$ and ${\xi:g^* g_* \rightarrow 1_{\calF}}$ respectively. It is well known that the unit and counit for $f^*\dashv f_*$ may be defined as the following composites
$$\xymatrix{
1_{\calS} \ar[r]^-{\alpha} & p_* p^* \ar[r]^-{p_* \nu_{p^*}} & p_* g_* g^* p^* = f_* f^* && 
  f^* f_* = g^* p^* p_* g_* \ar[r]^-{g^* \beta_{g_*}} & g^* g_* \ar[r]^-{\xi} & 1_{\calF}
}$$
so ${\oalpha = (p_* \nu_{p^*}) \alpha:1_{\calS} \rightarrow f_* f^*}$ and ${\obeta = \xi (g^* \beta_{g_*}):f^* f_* \rightarrow 1_{\calF}}$.

   (Notice that if $f^*$ and $p^*$ are fully faithful then ${p_* \nu_{p^*}}$ is forced to be an iso.)

   The next concept is well known and plays a relevant role here.

\begin{definition}\label{DefPreservationOfSindexedCopros} We say that ${g_*:\calF \rightarrow \calE}$  {\em preserves $\calS$-indexed coproducts} if the natural transformation 
${\nu_{p^*}:p^* \rightarrow g_* g^* p^* =g_* f^*}$ is an iso.  
\end{definition}

   When ${f^*:\calS \rightarrow \calF}$ and ${p^*:\calS \rightarrow \calE}$ are the discrete inclusions of the base $\calS$ into the respective (pre-)cohesive toposes, then preservation of $\calS$-indexed coproducts formalizes the idea that ${g_*:\calF \rightarrow \calE}$ preserves discrete spaces.

   Let us say that an object is {\em indecomposable} if it has exactly two complemented subobjects. In a Grothendieck topos $\calF$, an object $Z$  is  indecomposable if and only if ${\calF(Z, \_):\calF \rightarrow \Set}$ preserves coproducts. If the canonical ${f:\calF \rightarrow \Set}$ is essential then, $Z$ is indecomposable if and only if ${f_! Z = 1}$. In particular, if ${f}$ is pre-cohesive then $Z$ is indecomposable if and only if $Z$ is connected.

\begin{lemma}\label{lemma1.4} Let $\calC$ be a small category and let ${p:\Psh{\calC} \rightarrow \Set}$ be the associated presheaf topos. Let ${g:\calF \rightarrow \Psh{\calC}}$ be a geometric morphism. If for every $C$ in $\calC$, ${g^*(\calC(\_, C))}$ is indecomposable then ${g_*:\calF \rightarrow \Psh{\calC}}$ preserves $\Set$-indexed coproducts.
\end{lemma}

\proof By Yoneda and the adjunction $g^*\dashv g_*$ 
we have that, for every $X$ in $\calF$ and $C$ in $\calC$,  ${(g_* X) C \simeq \calF(g^*(\calC(-, C)), X)}$. In particular, for every indexed family ${(X_i \mid i\in I)}$ of objects in $\calF$, 
$(g_*(\sum_{i\in I}X_i)) C =\calF(g^*(\calC(-,C)),\sum_{i\in I}X_i)$. So, by hypothesis, 
\[
\bigg(g_*\bigg(\sum_{i\in I}X_i\bigg)\bigg)C \simeq \sum_{i\in I}\calF(g^*(\calC(-,C)),X_i)\simeq 
\sum_{i\in I}(g_*X_i)C=\bigg(\sum_{i\in I}g_*X_i\bigg)C
\]
which implies that $g_*$ preserves $\Set$-indexed coproducts.
\endproof

   As the referee has observed, the converse of Lemma~\ref{lemma1.4} holds in more generality: if $g_*$ preserves coproducts then $g^*$ preserves indecomposables.

\subsection{Assuming the existence of `pieces' functors}

   Assume from now on that both ${f^*:\calS \rightarrow \calF}$ and ${p^*:\calS \rightarrow \calE}$ have left adjoints denoted by ${f_!:\calF \rightarrow \calS}$ and ${p_!:\calE \rightarrow \calS}$ respectively. Denote the unit and counit of ${p_! \dashv p^*}$ by ${\sigma:1_{\calE} \rightarrow p^* p_!}$ and ${\tau:p_! p^* \rightarrow 1_{\calS}}$. Naturally, to maintain the consistency of our notation, we denote the unit and counit of ${f_! \dashv f^*}$ by ${\osigma:1_{\calF} \rightarrow f^* f_!}$ and ${\otau:f_! f^* \rightarrow 1_{\calS}}$.
   
   Before introducing the next piece of notation we stress that we are not assuming that ${g_*:\calF \rightarrow \calE}$ preserves $\calS$-indexed coproducts.

\begin{definition}[The canonical ${\varrho:f_! g^* \rightarrow p_!}$]  The transformation ${\nu_{p^*}:p^* \rightarrow g_* g^* p^* = g_* f^*}$ has a mate ${\varrho:f_! g^* \rightarrow p_!}$ which is defined by the pasting
$$\xymatrix{
\calE \ar@/_1pc/[rd]_-{id} \ar[r]^-{p_!} & \calS \ar[d]|-{p^*} \ar[r]^-{id} & \calS \ar[d]|-{g_* f^*} \ar@/^1pc/[rd]^-{id} & \\
\ar@{}[ru]|{\Rightarrow} & \calE \ar@{}[ru]|{\Rightarrow} \ar[r]_-{id} & \calE \ar@{}[ru]|{\Rightarrow} \ar[r]_-{f_! g^*} & \calS 
}$$
In other words, ${\varrho:f_! g^* \rightarrow p_!}$ is the composite
$$\xymatrix{
f_! g^* \ar[r]^-{f_! g^* \sigma} & f_! g^* p^* p_! \ar@/_1pc/[rrrrr]_-{id} \ar[rr]^-{f_! g^* (\nu_{p^* p_!})} && f_! g^* g_* g^* p^* p_!  \ar[r]^-{=} & 
f_! g^* g_* f^* p_! \ar[rr]^-{f_! (\xi_{f^* p_!})} & & f_! f^* p_! \ar[r]^-{\otau_{p_!}} & p_! 
}$$
or, more efficiently, $\xymatrix{f_! g^* \ar[r]^-{f_! g^* \sigma} & f_! g^* p^* p_! = f_! f^* p_! 
\ar[r]^-{\otau_{p_!}} & p_!}$. 
\end{definition}

The relation of the arrow $\varrho\pts f_!g^*\to p_!$ with preservation of $\calS$-indexed coproducts may be expressed as follows.
   
\begin{lemma}\label{LemIndexedCoprosAndPreservationOfPieces} The following are equivalent:
\begin{enumerate}
\item The functor ${g_*:\calF \rightarrow \calE}$ preserves $\calS$-indexed coproducts.
\item The transformation $\varrho\pts f_!g^*\to p_!$ is an isomorphism.
\item there exists a natural transformation ${\rho:p_! \rightarrow f_! g^*}$ such that the following diagrams 
$$\xymatrix{
g^* \ar[rrd]_-{\osigma_{g^*}} \ar[r]^-{g^* \sigma} & g^* p^* p_! \ar[r]^-{id} & f^* p_! \ar[d]^-{f^* \rho} && 
   p_! p^* \ar[rrd]_-{\tau}  \ar[r]^-{\rho_{p^*}} & f_! g^* p^* \ar[r]^-{id} & f_! f^* \ar[d]^-{\otau} \\
 & & f^* f_! g^* && & & 1_{\calS}
}$$
commute. 
\end{enumerate}
Moreover, in this case, $\rho$ is inverse to $\varrho$.
\end{lemma}
\begin{proof}
The first two items are equivalent because conjugation preserves isos.
If the second item holds then it is easy to check that ${\varrho^{-1}:p_! \rightarrow f_! g^*}$  
satisfies the conditions required by the third item.
 To complete the proof it is enough to show that any $\rho$ as in the third item  is actually inverse to $\varrho$. The following diagram
$$\xymatrix{
f_! g^*  \ar[rrd]_-{f_! \osigma_{g^*}} \ar[r]^-{f_! g^* \sigma} & f_! g^* p^* p_! \ar[r]^-{id} & f_! f^* p_! \ar[d]^-{f_! f^* \rho} \ar[r]^-{\otau} & p_! \ar[d]^-{\rho} \\ 
& & f_! f^* f_! g^* \ar[r]_-{\otau_{f_! g^*}} & f_! g^*
}$$
shows that ${\varrho}$ is a section of $\rho$. On the other hand, the next one
$$\xymatrix{
p_! \ar[d]_-{p_! \sigma} \ar[r]^-{\rho} & f_! g^* \ar[r]^-{f_! g^* \sigma} & f_! g^* p^* p_! \ar[r]^-{id} & f_! f^* p_! \ar[d]^-{\otau_{p_!}} \\
p_! p^* p_! \ar[rru]_-{\rho_{p^* p_!}} \ar[rrr]_-{\tau_{p_!}} & & & p_! 
}$$
shows that ${\rho}$ is a section of ${\varrho}$.
\end{proof}

Thus, in this restricted context, the concept of ${g_*:\calF \rightarrow \calE}$ preserving $\calS$-indexed coproducts can be reformulated by the idea, with `pieces' functors, that $g^*$ `preserves pieces'.

\begin{definition}
We say that ${p_!}$ {\em inverts the unit of $g$} if the natural ${p_! \nu:p_! \rightarrow p_! g_* g^*}$ is an iso. 
   Similarly,  we say that ${f_!}$ {\em inverts the counit of $g$} if the natural ${f_! \xi:f_! g^* g_* \rightarrow f_!}$ is an iso.
\end{definition}

We will need the following lemma regarding these concepts.

\begin{lemma}\label{LemRealizationPreservesPiecesAndInversionOfCounitImpliesInversionOfUnit}
If ${g_*:\calF \rightarrow\calE}$ preserves $\calS$-indexed coproducts, then ${p_!:\calE \rightarrow \calS}$ inverts the unit of $g$ if and only if ${f_! \xi_{g^*}:f_! g^* g_* g^* \rightarrow f_! g^*}$ is an iso.
Therefore, if ${g_*:\calF \rightarrow \calE}$ preserves $\calS$-indexed coproducts and $f_!$ inverts the counit of $g$ then  $p_!$ inverts the unit of $g$.
\end{lemma}
\begin{proof}
The following diagram 
$$\xymatrix{
f_! g^* \ar[d]_-{f_! g^* \nu} \ar[r]^-{\varrho} & p_! \ar[d]^-{p_! \nu} \\
f_! g^* g_* g^* \ar[r]_-{\varrho} & p_! g_* g^*
}$$
commutes by naturality. So, if ${\varrho}$ is  invertible then, one of the vertical maps is iso if and only if the other one is. Finally, notice that the left vertical map is an iso if and only if ${f_! \xi_{g^*}:f_! g^* g_* g^* \rightarrow f_! g^*}$ is an iso.
\end{proof}

   The following is one of the main definitions in the paper.

\begin{definition}\label{DefGpp}
We will say that the adjunction ${g:\calF \rightarrow \calE}$  {\em preserves pieces} if ${g_*:\calF \rightarrow \calE}$ preserves $\calS$-indexed coproducts and the composite
$$\xymatrix{
p_! g_* \ar[r]^-{\rho_{g_*}} & f_! g^* g_* \ar[r]^-{f_! \xi} & f_!
}$$
is an iso.
\end{definition}

Let us summarize the above as follows.

\begin{proposition}\label{PropObviousCharOfGpp}
If ${g_*:\calF \rightarrow \calE}$ preserves $\calS$-indexed coproducts then the following are equivalent:
\begin{enumerate}
\item The adjunction ${g:\calF \rightarrow \calE}$ preserves pieces.
\item The functor ${f_!:\calF \rightarrow \calS}$ inverts the counit of $g$.
\item There exists a natural iso ${\lambda:p_! g_* \rightarrow f_!}$ such that the following triangle
$$\xymatrix{
f_! g^* g_* \ar[rd]_-{f_! \xi} \ar[r]^-{\varrho_{g_*}} & p_! g_* \ar[d]^-{\lambda}  \\
 & f_! 
}$$
commutes. 
\end{enumerate}
Moreover, in this case, ${\lambda = (f_! \xi) (\rho_{g_*})}$, and ${p_!:\calE \rightarrow \calS}$ inverts the unit of $g$.
\end{proposition}
\begin{proof}
The first item implies the second because (by Lemma~\ref{LemIndexedCoprosAndPreservationOfPieces}) $\rho$ is an iso. The second implies the third using the explicit definition of $\lambda$. The third implies the the first because $\rho$ is inverse to $\varrho$. Finally, $p_!$ inverts the unit of $g$ by Lemma~\ref{LemRealizationPreservesPiecesAndInversionOfCounitImpliesInversionOfUnit}.
\end{proof}

   We will give a sufficient condition for ${g:\calF \rightarrow \calE}$ to preserve pieces.
   For that, we need to understand the natural transformations ${p_! g_* \rightarrow f_!}$ in more detail.

\begin{lemma}\label{LemLambdaVarrhoEqualsPiecesOfXi}
For any natural transformation ${\lambda:p_! g_* \rightarrow f_!}$ the triangle below commutes 
$$\xymatrix{
f_! g^* g_* \ar[rd]_-{f_! \xi} \ar[r]^-{\varrho_{g_*}} & p_! g_* \ar[d]^-{\lambda}  \\
 & f_! 
}$$
if and only if any of the two mate rectangles below
$$\xymatrix{
 g^* g_* \ar[d]_-{\xi} \ar[r]^-{g^* \sigma_{g_*}} & g^* p^* p_! g_* \ar[r]^-{id} & f^* p_! g_*  \ar[d]^-{f^* \lambda} &&
 g_* \ar[d]_-{g_* \osigma} \ar[r]^-{\sigma_{g_*}} & p^* p_! g_* \ar[r]^-{p^* \lambda} & p^* f_! \ar[d]^-{\nu_{p^* f_!}} \\
  1_{\calF} \ar[rr]_-{\osigma} && f^* f_! && g_* f^* f_! \ar[rr]_-{=}&& g_* g^* p^* f_! 
}$$
commutes.
\end{lemma}
\begin{proof}
We leave it the reader to check that 
\begin{enumerate}
\item The map ${g^* \sigma_{g_*}:g^* g_* \rightarrow g^* p^* p_! g_* = f^* p_! g_*}$ is the transpose of ${\varrho_{g_*}:f_! g^* g_* \rightarrow p_! g_*}$.
\item The composite
$$\xymatrix{
g^* g_* \ar[r]^-{\xi} & 1_{\calF} \ar[r]^-{\osigma} & f^* f_!
}$$
is the transpose of ${f_! \xi:f_! g^* g_* \rightarrow f_!}$.
\end{enumerate}
Then simply observe that the transpositions of the two maps ${f_! g^* g_* \rightarrow  f_!}$ in the triangle in the statement coincide with the corresponding maps in the left rectangle in the statement.
\end{proof}

\subsection{Further assuming discrete inclusions}

   In this section we assume further that ${p^*:\calS \rightarrow \calE}$ and ${f^*:\calS \rightarrow \calF}$ are full and faithful. In other words, we assume that ${\alpha:1_{\calS}  \rightarrow p_* p^*}$ and ${\oalpha:1_{\calS} \rightarrow f_* f^*}$ are isos. Recall that under the standing assumptions ${\oalpha = (p_* \nu_{p^*}) \alpha:1_{\calS} \rightarrow f_* f^*}$ so, ${p_* \nu_{p^*}:p_* p^* \rightarrow p_* g_* g^* p^* = f_* f^*}$ is forced to be an iso.

   In this context we have natural transformations ${\theta:p_* \rightarrow p_!}$ and ${\otheta:f_* \rightarrow f_!}$  defined as the following composites
$$\xymatrix{
p_* \ar[r]^-{p_* \sigma} & p_* p^* p_! \ar[r]^-{\alpha_{p_!}^{-1}} & p_! &&
  f_* \ar[r]^-{f_* \osigma} & f_* f^* f_! \ar[r]^-{\oalpha_{f_!}^{-1}} & f_! 
}$$
so, in more detail, the natural transformation ${\otheta:f_* \rightarrow f_!}$ is the following composite
$$\xymatrix{
f_* \ar[r]^-{f_* \osigma} & f_* f^* f_! \ar[r]^-{=} & p_* g_* g^* p^* f_! \ar[rr]^-{(p_* \nu_{p^* f_!})^{-1}} & & p_* p^* f_! \ar[r]^-{\alpha_{f_!}^{-1}} & f_!
}$$
using the description of $\oalpha$ in terms of $\alpha$ and $\nu$.

\begin{lemma}\label{LemLambdaTheta}
Let ${\lambda:p_! g_* \rightarrow f_!}$ be a natural transformation. Then the square on the left below commutes 
$$\xymatrix{
p_* g_* \ar[d]_-{=} \ar[r]^-{\theta_{g_*}} & p_! g_* \ar[d]^-{\lambda} && 
  p_* g_* \ar[d]_-{p_* g_* \osigma} \ar[r]^-{p_* \sigma_{g_*}} & p_* p^* p_! g_* \ar[r]^-{p_* p^* \lambda} & p_* p^* f_! \ar[d]^-{p_* \nu_{p^* f_!}} \\
f_* \ar[r]_-{\otheta} & f_! && 
  p_* g_* f^* f_! \ar[rr]_-{=} && p_* g_* g^* p^* f_!
}$$
if and only if the rectangle on the right above commutes.
\end{lemma}
\begin{proof}
Expanding the definitions of $\theta$ and $\otheta$ we obtain that commutativity of the left square in the statement is equivalent to commutativity of the outer rectangle below
$$\xymatrix{
p_* g_* \ar[d]_-{=} \ar[rrrr]^-{p_* \sigma_{g_*}} &&&& p_* p^* p_! g_* \ar[d]^-{p_* p^* \lambda} \ar[r]^-{\alpha_{p_! g_*}^{-1}} & p_! g_* \ar[d]^-{\lambda} \\
f_* \ar[r]_-{f_* \osigma} & f_* f^* f_! \ar[r]_-{=} & p_* g_* g^* p^* f_! \ar[rr]_-{(p_* \nu_{p^* f_!})^{-1}} & & p_* p^* f_! \ar[r]_-{\alpha_{f_!}^{-1}} & f_!
}$$
but since $\alpha$ is iso, the outer rectangle commutes if and only if the inner rectangle commutes. In turn, commutativity of the inner rectangle is equivalent to commutativity of the rectangle in the statement.
\end{proof}

   Notice that if ${\theta:p_* \rightarrow p_!}$ is epi then there is at most one ${\lambda:p_! g_* \rightarrow f_!}$ as in Lemma~\ref{LemLambdaTheta}.
   On the other hand, if ${\otheta:f_* \rightarrow f_!}$ is epi then any ${\lambda:p_! g_* \rightarrow f_!}$ as in Lemma~\ref{LemLambdaTheta} is epi.

\begin{lemma}\label{LemIfRealPpThenLambdaTheta}
If ${g_*:\calF \rightarrow \calE}$ preserves $\calS$-indexed coproducts
then  the canonical transformation  ${\lambda = (p_! \xi) (\rho_{g_*}):p_! g_* \rightarrow f_!}$ makes the following diagram
$$\xymatrix{
p_* g_* \ar[d]_-{=} \ar[r]^-{\theta} & p_! g_* \ar[d]^-{\lambda} \\
f_* \ar[r]_-{\otheta} & f_! 
}$$
commute.
\end{lemma}
\begin{proof}
The rectangle on the right of the statement of Lemma~\ref{LemLambdaTheta} coincides with the result of applying ${p_*:\calE \rightarrow \calS}$ to one of the rectangles of Lemma~\ref{LemLambdaVarrhoEqualsPiecesOfXi}.
\end{proof}

We can now prove our sufficient condition for the adjunction  $g$ to preserve pieces.
   
\begin{proposition}\label{cohlambda}
If the Nullstellensatz holds for ${p:\calE\rightarrow\calS}$ and ${g_*:\calF \rightarrow \calE}$ preserves $\calS$-indexed coproducts then, the adjunction ${g:\calF \rightarrow \calE}$ preserves pieces if and only if  and there exists a natural iso ${\lambda:p_! g_* \rightarrow f_!}$ such that the following diagram
$$\xymatrix{
p_* g_* \ar[d]_-{=} \ar[r]^-{\theta} & p_! g_* \ar[d]^-{\lambda} \\
f_* \ar[r]_-{\otheta} & f_! 
}$$
commutes. Moreover, in this case, ${\lambda = (p_! \xi) (\rho_{g_*}):p_! g_* \rightarrow f_!}$.
\end{proposition}
\begin{proof}
First notice that if the Nullstellensatz holds for $p$ then there exists at most one $\lambda$ as in the statement. Hence, if ${g_*:\calF\rightarrow \calE}$ preserves
$\calS$-indexed coproducts, Lemma~\ref{LemIfRealPpThenLambdaTheta} implies that such $\lambda$ must be the canonical one. Finally, by definition, this $\lambda$ is an iso if and only the adjunction $g$ preserves pieces.
\end{proof}

Notice that if we also assume that the Nullstellensatz holds for ${f:\calF \rightarrow \calS}$, and $\calS$ is
balanced, then it is enough to require that $\lambda$ be mono.

  In the presence of connectors we have the following.

\begin{lemma}\label{lemma1.3}
Assume that $g_*\pts\calF\to\calE$ preserves
$\calS$-indexed coproducts and that $0,1\pts 1\to I$ is a connector for $p\pts\calE\to\calS$. If
$g^*0,g^*1\pts 1\to g^*I$ is a connector for $f\pts\calF\to\calS$, then $g\pts\calF\to\calE$ preserves pieces.
\end{lemma}

\proof Our hypothesis imply that the two horizontal forks below are coequalizers
\[
\xy
(-10,8)*+{p_*((g_*F)^I)}="1",
(30,8)*+{p_*(g_*F)}="2",
(60,8)*+{p_!(g_*F)}="3",
(-10,-8)*+{f_*(F^{g^*I})}="1b",
(30,-8)*+{f_*F}="2b",
(60,-8)*+{f_!F}="3b",
\POS "1" \ar@<3pt>^{p_*\ev_0} "2",
\POS "1" \ar@<-2pt>_{p_*\ev_1} "2",
\POS "2" \ar^{\theta} "3",
\POS "1b" \ar@<3pt>^{f_*\ev_{g^*0}} "2b",
\POS "1b" \ar@<-2pt>_{f_*\ev_{g^*1}} "2b",
\POS "2b" \ar^{\overline{\theta}} "3b",
\POS "1" \ar_{\simeq} "1b",
\POS "2" \ar_{=} "2b",
\POS "3" \ar@{.>}^{\lambda_F} "3b",
\endxy
\]
where the left arrow is the canonical iso. It is straightforward to check that the left part of the diagram commutes sequentially, and
thus induces a unique isomorphism $\lambda_F$ as shown above. Thus the result follows from Proposition \ref{cohlambda}.
\endproof

\section{Some geometric morphisms that preserve pieces}
\label{SecExamplesOfPpGm}

   Let $\calC$ be a small category with terminal object and such that every object has a point so that ${p:\Psh{\calC} \rightarrow \Set}$ is pre-cohesive.
Let ${g:\calF \rightarrow \Psh{\calC}}$ be a geometric morphism such that the composite ${f = p g:\calF\rightarrow \Set}$ is pre-cohesive.
We can combine Lemmas~\ref{lemma1.3} and \ref{lemma1.4} as follows.

\begin{proposition}\label{proposition1.5}
Let $0,1\pts 1\to I$ be a connector for the pre-cohesive $p\pts\widehat{\calC}\to\con$ such that $g^*0,g^*1\pts 1\to g^*I$
is a connector for $f\pts\calF\to\con$. If ${g^* I}$ is connected and for every $C$ in $\calC$ there exists a finite set $S$
such that $\calC(-,C)$ is a retract of $I^{p^*S}$ in $\widehat{\calC}$, then $g\pts\calF\to\widehat{\calC}$ preserves pieces.
\end{proposition}

\proof By Lemma \ref{lemma1.3} it suffices to show that $g_*$ preserves $\con$-indexed
coproducts. In turn, by Lemma \ref{lemma1.4},
it suffices to show that $g^*(\calC(-,C))$ is connected for every $C\in\calC$. Since retracts of connected objects are
connected, it suffices to show that $g^*(I^{p^*S})$ is connected for any finite set $S$. Since $g^*$ preserves finite products,
$g^*(I^{p^*S})\simeq (g^*I)^{g^*(p^*S)}=(g^*I)^{f^*S}$. As $f_!$ preserves finite products, $(g^*I)^{f^*S}$ is connected if
$g^*I$ is.
\endproof

In fact, one could weaken the hypothesis by requiring only that $g^*(\calC(-,C))$ is a retract of $(g^*I)^{f^*S}$, but
the above is enough for our purposes.

\begin{corollary}\label{corollary1.6}
Let $\calC$ have finite products and let $0,1\pts 1\to I$ be a bipointed object in $\calC$ such that every
object of $\calC$ is a retract of a finite power of $I$. If $0,1\pts 1\to I$ (as a bipointed object in $\widehat{\calC})$ is a 
connector for $p\pts\widehat{\calC}\to\con$, $g^*0,g^*1\pts 1\to g^*I$ is a connector for $f\pts\calF\to\con$
and $g^*I$ is connected, then $g\pts\calF\to\widehat{\calC}$ preserves pieces.
\end{corollary}

For a concrete example let us consider the cohesive topos $f\pts\sh(\bflp,K)=\calF\to \con$ of picewise linear maps
from \cite{Menni2014a} that we mentioned above. It is shown loc.\ cit.\ that the representable $I'=\bflp(-,[0,1])$ is a 
connector for $f$ and may be equipped with a total order with distinct endpoints. Hence $I'$ is, in particular, a model for the 
theory of strictly bipointed objects.

\begin{proposition}\label{PropBipointed}
The geometric morphism $g\pts\calF\to\con^\calA$ such that $g^*(\calA(I,-))=I'$ is surjective and preserves pieces.
\end{proposition}
\proof The category $\calA^{\textup{\scriptsize{op}}}$ is a Lawvere theory and, in particular, every object is a power of $I$. So,
in order to apply Corollary~\ref{corollary1.6} we need only prove that $I$ in $\calA^{\textup{\scriptsize{op}}}$ induces a connector in
$\con^\calA$. We leave the details to the reader. (One possible argument is to use Lemma 8.9 in \cite{Menni2014a}.)

   To prove that $g$ is surjective let ${c:\Set \rightarrow \calF}$ be the geometric morphism such that ${c_* = f^!:\Set \rightarrow \calF}$. Then, as in Corollary~\ref{CorJohnstone}, it is enough to prove that the composite geometric morphism ${g c:\Set \rightarrow \con^\calA}$ is surjective. Now, the inverse image of ${g c}$ sends ${I}$ in ${\con^\calA}$ to ${c^* (g^* I) = f_* (g^* I) = [0, 1]}$. So ${g c}$ is surjective by Corollary~\ref{bipointedsurjection}.
\endproof

Proposition~{10.6} in \cite{Menni2014a}  shows the existence of a geometric morphism ${g\pts\calF\to\widehat{\Delta}}$ such that ${g^* I}$ is a connector.   

\begin{proposition}\label{PropSimplicial}
The geometric morphism $g\pts\sh(\bflp,K)\to\widehat{\Delta}$ is surjective and preserves pieces.
\end{proposition}

\proof (Notice that we can not apply Corollary \ref{corollary1.6} in this case since $\Delta$
does not have finite products.) In order to apply Proposition \ref{proposition1.5} it only remains to show that every representable object in 
$\widehat{\Delta}$ is a retract of a finite power of $\Delta(-,[1])$. This is surely well-known but
we have been unable to find an appropriate reference so we sketch a proof.

The inclusion $\Delta\to\cate$ induces a functor $\cate\to\widehat{\Delta}$ with a left adjoint, 
and it is shown in Corollary 4.3 in \cite{GabrielZisman} that this functor is full and faithful.
Now, for each $n\in\mathbb{N}$ the product $[1]^n$ in 
$\cate$ may
be identified with the Boolean algebra of parts of a set with $n$ elements. It is then clear that the ``cardinality'' map
$[1]^n\to [n]$ in $\cate$ is a retraction for any maximal chain $[n]\to [1]^n$. 
The embedding $\cate\to\widehat{\Delta}$ preserves this retract
showing that $\Delta(-,[n])$ is a retract of $\Delta(-,[1])^n$. 

Another possibility is to construct explicitly such a retraction. Indeed, notice that for every $j\in\{1,\dots,n\}$, we have 
the map $a_j\pts[n]\to [1]$ in $\Delta$ such that $a_j(i)=0$ for all $i<j$ and $a_j(i)=1$ for all $i\geq j$. The resulting
family $\langle \Delta(-,a_j)\rangle_{j=1}^n$ of maps of $\widehat{\Delta}$ determines a unique map $a\pts
\Delta(-,[n])\to \Delta(-,[1])^n$ to the product. In the opposite direction we define a map $b\pts \Delta(-,[1])^n\to
\Delta(-,[n])$ in $\widehat{\Delta}$ such that for every $[m]\in \Delta$, each family
$\langle h_j\pts[m]\to[1]\rangle_{j=1}^n$ in $\Delta([m],[1])^n$ and each ${i\in [m]}$,
\[
b_{[m]}(\langle h_j\pts[m]\to[1]\rangle_{j=1}^n)(i) =\sum_{j=1}^nh_j(i).
\]
We leave the details to the reader.

Surjectivity is proved as in Proposition~\ref{PropBipointed}, but using Corollary~\ref{CorFolklore} instead of~\ref{bipointedsurjection}.
\endproof

   An analogue of Propositions~\ref{PropBipointed} and \ref{PropSimplicial} holds for the classifier of connected distributive lattices (Example~{8.11} in \cite{Menni2014a}) and surely for many other classifiers for theories for which the unit interval is a model.

   It is relevant to mention that the condition ${p_! g_* = f_!}$ appears in Lemma~{2.2.16} in \cite{BungeFunkBook} for the case when ${g:\calF \rightarrow \calE}$ is a subtopos. Indeed, while our motivation comes form surjective examples, preservation of pieces is independent of surjectivity and there are natural examples of inclusions that preserve pieces. In this case, the intuition that the left adjoint is a `geometric realization' might need some adjustment.

   For instance, let $n\geq 1$ be an integer, and consider the inclusion ${i:\Delta_n\to\Delta}$ of the full subcategory of 
$\Delta$ that consists of
all those $[m]=\{0,\dots,m\}$ with $m\leq n$. This inclusion induces the geometric morphism
${g:\Psh{\Delta_n} \rightarrow \Psh{\Delta}}$  whose direct image functor is the coskeleton functor ${g_*=\cosk^n\pts \Psh{\Delta_n} \to\Psh{\Delta}}$, whereas the
inverse image is restriction of a simplicial set to $\Delta_n$, ${g^* = \trn : \Psh{\Delta} \rightarrow \Psh{\Delta_n}}$. It is well known that
this geometric morphism is essential (the extra left adjoint is the skeleton functor ${\skn:\Psh{\Delta_n} \rightarrow \Psh{\Delta}}$) and an
embedding. Now, both canonical ${f: \Psh{\Delta_n} \rightarrow \Set}$ and ${p:\Psh{\Delta} \rightarrow \Set}$ are pre-cohesive and we may assume that ${p g = f}$.

\begin{corollary}\label{CorTruncations} For any ${n\geq 1}$, ${g:\Psh{\Delta_n} \rightarrow \Psh{\Delta}}$ preserves pieces.
\end{corollary}
\begin{proof}
Follows from Proposition~\ref{proposition1.5}. Indeed,  the calculations in Proposition~\ref{PropSimplicial} show that every representable in $\Psh{\Delta}$ is a retract of a power of the connector ${\Delta(\_, [1]) = I}$ in $\Psh{\Delta}$. It remains to show that ${g^* I = \Delta_n(\_,[1])}$ is a connector in ${\Psh{\Delta_n}}$, but this holds just as in the case of ${\Psh{\Delta}}$; see Example~{8.10} in  \cite{Menni2014a}. Incidentally, using essentially the same idea, one may prove that ${\Delta_n \rightarrow \Delta}$ is cofinal in the sense of IX.3 in \cite{maclane}, so  ${f_! g^* \cong p_!}$ because $g^*$ is restriction along ${\Delta_n \rightarrow \Delta}$ and ${f_!}$, ${p_!}$ take colimits.
\end{proof}

    Regardless of connectors, when $g_*$ is full and faithful the second condition of Proposition~\ref{PropObviousCharOfGpp} is trivially satisfied, so we may conclude the following.

\begin{corollary}\label{Corgsubstarff}
Let ${p:\calE \rightarrow \calS}$ be pre-cohesive and let ${g:\calF \rightarrow \calE}$ be a subtopos such that ${f = p g:\calF \rightarrow \calS}$ is pre-cohesive. Then $g$  preserves pieces if and only if ${g_*:\calF \rightarrow \calE}$ preserves $\calS$-indexed coproducts.
\end{corollary}

   In the examples of Corollary~\ref{CorTruncations} both the domain and codomain of the relevant geometric morphism are presheaf toposes. In contrast, consider the next result involving locally connected coverages (Section~{C3.3} in \cite{elephant}).

\begin{lemma}
Let $\calC$ be a small category with terminal object and such that every object has a point, so that $p\pts\widehat{\calC}\to
\con$ is precohesive. Let $(\calC,J)$ be a site and denote the inclusion by $g\pts\sh(\calC,J)\to\widehat{\calC}$ and the
composite $pg\pts\sh(\calC,J)\to\con$ by $f$. If the coverage $J$ is locally connected, then $g$ preserves pieces.
\end{lemma}

\proof The geometric morphism $f\pts\sh(\calC,J)\to\con$ is pre-cohesive by results in \cite{Johnstone2011}. 
By Corollary~\ref{Corgsubstarff} all we have to do is show that $g_*\pts\sh(\calC,J)\to\widehat{\calC}$
preserves $\con$-indexed coproducts, but notice that this fact
is perhaps more familiar than over an arbitrary base: ``local connectedness of the site implies that constant presheaves on $\calC$ are $J$-sheaves''; Propositon~{1.3} loc.\ cit..
\endproof

   For instance, the inclusion ${\sh(\bflp, K) \rightarrow \Psh{\bflp}}$ preserves pieces (over $\Set$). Something analogous happens to the other examples of cohesive toposes in \cite{Menni2014a}.

\section{The preservation of weakly Kan objects}
\label{SecPPandWK}

   The comparison between Serre fibrations and Kan fibrations allows Gabriel and Zisman to conclude that the singular complex of any topological space is a Kan complex. (See~{VII.1.6} in \cite{GabrielZisman}.) We prove an analogous result in this section.

   Let ${f:\calF \rightarrow \calS}$ and ${p:\calE \rightarrow \calS}$ be pre-cohesive toposes. Let ${g:\calF\rightarrow \calE}$ be such that ${p g  = f:\calF \rightarrow \calS}$.      
   We will show that if $g$ preserves pieces and $K$ is a  weakly Kan object  in $\calF$ then ${g_* K}$ is weakly Kan in $\calE$. Disregarding coherence for a moment, the argument is easy to sketch. Indeed, if we let ${A \in  \calS}$, we have that
\[ p_!((g_* X)^{p^* A}) = p_!(g_*(X^{g^*(p^* A)})) = f_!(X^{f^* A}) = (f_! X)^A = (p_! (g_* X))^A  \]
suggesting that ${g_*:\calF \rightarrow \calE}$ preserves weakly Kan objects. 
   Now, concerning an actual proof, the main coherence fact is proved in Lemma~\ref{LemKappaOkappa} below. Before that, we state some auxiliary facts.
   
\begin{lemma}\label{LemPpiecesPreservesProds}
 If ${g^*:\calE \rightarrow \calF}$ preserves pieces then the following diagram 
$$\xymatrix{
p_! X \times p_! Y \ar[rrr]^-{\rho \times \rho} & & & f_!(g^* X) \times f_!(g^* Y) \\
p_!(X\times Y)\ar[u]^-{\twopl{p_! \pi_0}{p_!  \pi_1}} \ar[r]_-{\rho} & f_!(g^*(X\times Y)) \ar[rr]_-{f_!\twopl{g^*\pi_0}{g^* \pi_1}} & & f_!((g^* X) \times (g^* Y)) \ar[u]_-{\twopl{f_! \pi_0}{f_!  \pi_1}} \\
}$$
commutes for every $X$ and $Y$ in $\calE$.
\end{lemma}
\begin{proof}
This simple fact is left for the reader.
\end{proof}

   Recall that we denoted by ${\kappa = \kappa^{p_!}_{X, Y}:p_!(Y^X) \rightarrow (p_! Y)^{p_! X}}$ the canonical natural transformation determined by the product preserving ${p_!:\calE\rightarrow \calS}$.

\begin{lemma}\label{LemKappa} The following diagram 
$$\xymatrix{
(p_! Y)^{(p_! X)} \times p_! X \ar[rr]^-{ev} && p_! Y \\
p_!(Y^X) \times p_! X \ar[u]^-{\kappa_{X, Y} \times id}  && \ar[ll]_-{\cong}^-{\twopl{p_! \pi_0}{ p_! \pi_1}} p_! (Y^X \times X)  \ar[u]_-{p_! ev}
}$$
commutes.
\end{lemma}

   Similarly, since ${f_!:\calF \rightarrow \calS}$ preserves finite products then there exists, for every $X$ and $Y$ in $\calF$, a unique ${\okappa = \kappa^{f_!}_{X, Y}:f_!(Y^X) \rightarrow (f_! Y)^{(f_! X)}}$ making an analogous diagram commute.

   The next result implies the well known fact that geometric morphisms are canonically enriched in their codomains.
   
\begin{lemma}\label{LemGisStrong}
There exists a unique natural ${\gamma:g^*((g_* F)^E) \rightarrow F^{g^* E}}$ such that the following diagram
$$\xymatrix{
g^*((g_* F)^E) \times g^* E \ar[rr]^-{\gamma \times id} & & F^{g^* E} \times g^* E \ar[r]^-{ev} & F \\
g^*((g_* F)^E \times  E)\ar[u]_-{\cong}^-{\twopl{g^* \pi_0}{g^* \pi_1}}
\ar[rrr]_-{g^* ev} &&& g^*(g_* F) \ar[u]_-{\xi}
}$$
commutes. Moreover, the composite
$$\xymatrix{
(g_* F)^E \ar[r]^-{\nu} & g_* (g^* ((g_* F)^E)) \ar[r]^-{g_* \gamma} & g_*(F^{g^* E})
}$$
is an iso. 
\end{lemma}
\begin{proof}
Preservation of binary products allows the following calculation with natural isos
\[ \calE(X, (g_* F)^E) \cong \calE(X\times E, g_* F) \cong \calF(g^*(X\times E), F) \cong \] \[ \cong \calF(g^* X \times g^* E, F) \cong \calF(g^* X, F^{g^* E}) \cong \calE(X, g_*(F^{g^* E})) \]
so, if we take ${X = (g_* F)^E}$, take the identity on it, and follow the instructions given by the above calculation then we get an iso ${(g_* F)^E \rightarrow g_*(F^{g^* E})}$.
In more detail, the natural iso ${\calE(X, (g_* F)^E) \cong \calF(g^* X, F^{g^* E})}$ sends the the identity on ${X = (g_* F)^E}$ to the unique map ${\gamma:g^*((g_* F)^E) \rightarrow  F^{g^* E}}$ such that the diagram in the statement commutes. The natural iso ${\calF(g^* X, F^{g^* E}) \cong \calE(X, g_*(F^{g^* E}))}$ sends $\gamma$ to the composite in the last part of the statement. As we have already mentioned, this composite must be an iso.
\end{proof}

   We emphasize the following.
   
\begin{lemma} If  ${g:\calF \rightarrow \calE}$ preserves pieces then the following composite 
$$\xymatrix{
p_!((g_* F)^E) \ar[r]^-{\rho} 
 & f_! (g^* ((g_* F)^E)) \ar[rr]^-{f_! \gamma} & & f_!(F^{g^* E})
}$$
is an iso.
\end{lemma}
\begin{proof}
Just notice that the following diagram
$$\xymatrix{
p_!((g_* F)^E) \ar[rd]_-{\rho} \ar[r]^-{p_! \nu} & p_!(g_* (g^* ((g_* F)^E))) \ar[d]_-{\lambda} \ar[rr]^-{p_! (g_* \gamma)} && p_!(g_*(F^{g^* E})) \ar[d]^-{\lambda} \\
 & f_! (g^* ((g_* F)^E)) \ar[rr]_-{f_! \gamma} & & f_!(F^{g^* E})
}$$
commutes.
\end{proof}

   The above simple facts were stated to efficiently prove the next result.

\begin{lemma}\label{LemKappaOkappa}
 If   ${g:\calF \rightarrow \calE}$ preserves pieces then the following diagram
$$\xymatrix{
p_!((g_* F)^E) \ar[d]_-{\kappa} \ar[r]^-{\rho}  & f_! (g^* ((g_* F)^E)) \ar[rr]^-{f_! \gamma} & & f_!(F^{g^* E}) \ar[d]^-{\okappa} \\
(p_!(g_* F))^{p_! E} \ar[rrr]_-{\lambda^{\varrho}} &  & & (f_! F)^{f_!(g^* E)}
}$$
commutes, for every $E$ in $\calE$ and $F$ in $\calF$. (Notice that the horizontal maps are isos.)
\end{lemma}
\begin{proof}
We can consider the transpositions ${(p_!(g_* F))^{p_! E} \times f_!(g^* E) \rightarrow f_! F}$ and we will find it useful to pre-compose with the composite
$$\xymatrix{
p_!((g_* F)^E) \times E)  \ar[rr]^-{\twopl{p_! \pi_0}{p_! \pi_1}} &&p_!((g_* F)^E) \times p_! E \ar[r]^-{id\times \rho} & (p_!(g_* F))^{p_! E} \times f_!(g^* E)
}$$
which is obviously an iso. If we do this to the left-bottom composite of the rectangle in the statement then we get the following diagram
$$\xymatrix{
p_!((g_* F)^E) \times p_! E  \ar[rd]|-{\kappa\times id} \ar[r]^-{\kappa\times\rho}  & (p_!(g_* F))^{p_! E} \times f_! (g^* E) \ar[d]_-{id\times\varrho} \ar[rr]^-{\lambda^{\varrho} \times id} & & (f_! F)^{f_! (g^* E)} \times f_! (g^* E) \ar[d]^-{ev} \\ 
p_!((g_* F)^E) \times E) \ar@/_1.5pc/[rr]_-{p_! ev} \ar[u]^-{\twopl{p_! \pi_0}{p_! \pi_1}} & p_!((g_* F)^E) \times p_! E \ar[r]^-{ev} & p_!(g_* F) \ar[r]^-{\lambda} & f_! F \\
 }$$
using Lemma~\ref{LemKappa}. On the other hand, we first observe that the following diagram
$$\xymatrix{
f_!(g^*((g_* F)^E)) \times f_!(g^* E) \ar[r]^-{f_! \gamma \times id} & f_!(F^{g^* E}) \times f_!(g^* E) \ar[r]^-{\okappa \times id} & (f_! F)^{f_! (g^* E)} \times f_!(g^* E) \ar[d]^-{ev} \\
f_!(g^*((g_* F)^E) \times g^* E) \ar[u]_-{\twopl{f_! \pi_0}{f_! \pi_1}} \ar[r]_-{f_!(\gamma\times id)} & f_!(F^{g^* E} \times g^* E) \ar[u]_-{\twopl{f_! \pi_0}{f_! \pi_1}}  \ar[r]_-{f_! ev} & f_! F
}$$
commutes by the analogue of Lemma~\ref{LemKappa} for $\okappa$.  So, taking ${X = (g_* F)^E}$ and ${Y = E}$ in Lemma~\ref{LemPpiecesPreservesProds}, and pasting with the previous diagram, it is clear that it only remains to observe that the following diagram
$$\xymatrix{
p_!((g_* F)^E \times E) \ar[dd]_-{p_! ev} \ar[r]^-{\rho} & f_!(g^*((g_* F)^E \times E)) \ar[rr]^-{f_!\twopl{g^*\pi_0}{g^* \pi_1}} \ar[dd]^-{f_!(g^* ev)} & & f_!(g^*((g_* F)^E) \times g^* E)  \ar[d]^-{f_!(\gamma\times id)} \\
& & & f_!(F^{g^* E} \times g^* E)  \ar[d]^-{f_! ev} \\
p_!(g_* F) \ar@/_1pc/[rrr]_-{\lambda} \ar[r]^-{\rho} & f_!(g^* (g_* F)) \ar[rr]^-{f_! \xi} &  & f_! F
}$$
commutes, using Lemma~\ref{LemGisStrong}.
\end{proof}

   We can now prove the promised result.
   
\begin{proposition}\label{PropDirectImagePreservesWK}
If ${g:\calF \rightarrow \calE}$ preserves pieces then ${g_*:\calF \rightarrow \calE}$ preserves weakly Kan objects. 
\end{proposition}
\begin{proof}
Let $F$ be a weakly Kan object  in $\calF$.
To prove that ${g_* F}$ is weakly Kan in $\calE$ let $A$ in $\calS$ and take ${E = p^* A}$ in Lemma~\ref{LemKappaOkappa} to obtain the following commutative diagram
$$\xymatrix{
p_!((g_* F)^{p^* S}) \ar[d]_-{\kappa} \ar[r]^-{\rho}  & f_! (g^* ((g_* F)^{p^* S})) \ar[rr]^-{f_! \gamma} & & f_!(F^{g^* {p^* S}}) \ar[d]^-{\okappa} \ar[r]^-{=} & f_!(F^{f^* S}) \ar[d]^-{\okappa} \\
(p_!(g_* F))^{p_! {p^* S}} \ar[rrr]_-{\lambda^{\varrho}} &  & & (f_! F)^{f_!(g^* {p^* S})} \ar[r]^-{=} & (f_! F)^{f_!(f^*  S)}
}$$
with iso horizontal maps. Since the right vertical map is an iso then so is the left vertical map, and this means that ${g_* F}$ is weakly Kan.
\end{proof}

   We do not know a simple sufficient condition for ${g^*:\calE \rightarrow \calF}$ to preserve weakly Kan objects; but notice that, in our examples, this is typically the case for the trivial reason that every object in $\calF$ is weakly Kan. In this case, we have a composite functor
$$\xymatrix{
\calE \ar[r]^-{g^*} & \calF \ar[r]^{g_*} & \wk\calE
}$$
that preserves finite products.

\section{The passage to homotopy}
\label{SecHomotopy}

   In Chapter~{IV} of \cite{GabrielZisman}, Grabriel and Zisman introduce the category of simplicial sets ``modulo homotopy" (denoted by ${\overline{\Delta^{\circ}\calE}}$) equipped with a calculus of fractions given by the {\em anodyne extensions}.  The associated category of fractions is called the {\em homotopy category} and it is denoted by $\calH$. It is then proved that the canonical ${\overline{\Delta^{\circ}\calE} \rightarrow \calH}$ has a right adjoint ${\calH \rightarrow \overline{\Delta^{\circ}\calE}}$ which induces an equivalence between $\calH$ and the category of Kan complexes modulo homotopy (see {IV.3.2.1} loc.\ cit.).

   In this section we show that, except for the calculus of fractions, a similar picture appears in the context given by a pieces-preserving geometric morphism from a cohesive topos to a pre-cohesive one.
   It will be convenient for the reader to be acquainted with the rudiments of enriched category theory \cite{kelly}. We recall here some of the basic facts that we need in beginning.

   As explained after B2.1.7 in \cite{elephant}, a product preserving functor ${F:\calV \rightarrow \calW}$ induces a 2-functor ${F_{\bullet}:\enrichedIn{\calV} \rightarrow \enrichedIn{\calW}}$, 
where $\enrichedIn{\calV}$ and $\enrichedIn{\calW}$ are the 2-categories of $\calV$-enriched and $\calW$-enriched categories respectively. For any $\calV$-enriched category $\calC$, ${F_{\bullet} \calC}$ has the same objects as $\calC$ but, for ${C_0}$ and ${C_1}$ in $\calC$, ${(F_{\bullet} \calC)(C_0, C_1) = F(\calC(C_0, C_1))}$, and composition and identities are obtained by applying $F$ to those of $\calC$. It is relevant to stress that $\calC$ and ${F_{\bullet}\calC}$ do not in general have the same underlying (ordinary) category.

   The comparison maps ${\kappa_{U, V}:F(V^U) \rightarrow (F V)^{(F U)}}$ give rise to a $\calW$-functor ${\breve{F}:F_{\bullet} \calV \rightarrow \calW}$.

\begin{lemma}\label{LemMenniFolklore}
If ${F:\calV \rightarrow \calW}$ has a full and faithful right adjoint then ${\breve{F}:F_{\bullet} \calV \rightarrow \calW}$ has a fully faithful $\calW$-enriched right adjoint.
\end{lemma} 
\begin{proof}
Let ${R:\calW \rightarrow \calV}$ be the right adjoint to $F$   (with unit $\sigma$ and counit $\tau$). The right adjoint to $\breve{F}$ will be denoted by ${\breve{R}:\calW \rightarrow F_{\bullet} \calV}$. On objects it is simply $R$, whereas for ${W, X}$ in $\calW$, ${\breve{R}_{W, X}:\calW(W, X) \rightarrow (F_{\bullet} \calV)(\breve{R} W, \breve{R} X) = F(\calV(R W, R X))}$ is the map ${X^W \rightarrow F((R X)^{(R W)})}$ given by the composite
$$\xymatrix{
X^W \ar[r]^-{\tau^{-1}} & F (R (X^W)) \ar[rr]^-{F(\kappa^R_{W, X})} && F((R X)^{(R W)})
}$$ 
in $\calW$. By Corollary~{A1.5.9} in \cite{elephant}, ${\kappa^R_{W, X}:R (X^W) \rightarrow (R X)^{(R W)}}$ is an iso so ${\breve{R}:\calW \rightarrow F_{\bullet} \calV}$ is fully faithful as a $\calW$-functor. The $\calW$-natural iso ${\calW(\breve{F} V, W) \cong (F_{\bullet} \calV)(V, \breve{R} W)}$ is determined by the iso
$$\xymatrix{
W^{F V} \ar[r]^-{\tau^{-1}} & F(R(W^{F V})) \ar[rr]^-{F(\kappa^R_{F V, W})} && F((R W)^{R(F V)}) \ar[rr]^-{F((R W)^{\sigma})} && F((R W)^V)
}$$
with inverse 
$$\xymatrix{
F((R W)^V) \ar[r]^-{\kappa} & (F ( R W))^{F V} \ar[r]^-{\tau^{F V}} & W^{F V}
}$$
in $\calW$.
\end{proof}

\subsection{The Hurewicz category of a (pre-)cohesive topos}
\label{SecHurewicz}

   Paraphrasing the text following Axiom 1 in \cite{Lawvere86}, the requirement that ${p_!:\calE \rightarrow \calS}$ preserves finite products is necessary for the naive construction of the homotopic passage from quantity to quality; namely, it insures that (not only $p_*$ but also) ${p_!:\calE \rightarrow \calS}$ is a closed functor, thus inducing a second way of associating an $\calS$-enriched category to each $\calE$-enriched category
$$\xymatrix{
\enrichedIn{\calE} \ar[rr]<+.5ex>^-{(p_!)_{\bullet}} \ar[rr]<-.5ex>_-{(p_*)_{\bullet}} & & \enrichedIn{\calS}
}$$
that we will denote by ${\bfh = (p_!)_{\bullet} : \enrichedIn{\calE} \rightarrow \enrichedIn{\calS}}$.
So, for example, $\calE$ itself as an $\calE$-enriched category gives rise to the {\em Hurewicz category} ${\bfh\calE = (p_!)_{\bullet}\calE}$. Its objects are those of $\calE$, and for each $X$, $Y$ in ${\bfh\calE}$, ${(\bfh\calE)(X, Y) = p_!(Y^X)}$. 
   
   As suggested by intuition, in the cases of main interest ${\calE}$ and ${\bfh\calE}$ will not have the same underlying (ordinary) category. The functor ${\breve{p_!}:\bfh\calE \rightarrow \calS}$ sends a `homotopy type' to the associated `set' of pieces. Moreover, the adjunction ${p_! \dashv p^*:\calS \rightarrow \calE}$ satisfies the hypotheses of Lemma~\ref{LemMenniFolklore} so it induces a $\calS$-enriched adjunction ${\breve{p_!}\dashv \breve{p^*}}$ with fully faithful ${\breve{p^*}:\calS \rightarrow \bfh{\calE}}$.
   In other words, as expected, the homotopy type of a discrete space is discrete.

   On the other hand, ${(p_*)_{\bullet}\calE}$ is just the canonical $\calS$-enrichment of $\calE$ given by the geometric morphism $p$, so we may denote it by $\calE$. The composite
$$\xymatrix{
((p_*)_{\bullet}\calE)(X, Y) = p_*(Y^X) \ar[r]^-{\theta} & p_!(Y^X) \ar[r]^-{\kappa} & (p_! Y)^{(p_! X)} 
}$$
underlies an $\calS$-enriched functor ${p_!:(p_*)_{\bullet}\calE \rightarrow \calS}$, with $\calS$ considered as enriched in itself.  
   
  Now, the natural transformation ${\theta:p_* \rightarrow p_!}$ induces an $\calS$-functor ${\calE \rightarrow \bfh\calE}$. Intuitively, it assigns to each space $X$ in $\calE$ the associated `homotopy type'. 
   Moreover, the diagram on the left below commutes
$$\xymatrix{
(p_*)_{\bullet} \calE \ar[rd]_-{p_!} \ar[r] & (p_!)_{\bullet} \calE \ar[d]^-{\breve{p_!}} && \calE \ar[rd]_-{p_!} \ar[r]& \bfh\calE \ar[d]^-{\breve{p_!}} \\
 & \calS && & \calS
}$$
 in $\enrichedIn{\calS}$ or, in a friendlier notation, the diagram on the right above commutes.

\begin{lemma}\label{LemWKisQuintessential} The $\calS$-adjunction ${\breve{p_!} \dashv \breve{p^*}:\calS \rightarrow \bfh\calE}$ restricts to one
${\breve{p_!} \dashv \breve{p^*}:\calS \rightarrow \bfh(\wk{\calE})}$ and, moreover, this restriction is  `quintessential' in the sense that the fully faithful ${\breve{p^*}:\calS \rightarrow \bfh(\wk{\calE})}$ is also left adjoint to ${\breve{p_!}:\bfh(\wk{\calE}) \rightarrow \calS}$.
\end{lemma}
\begin{proof}
The fact that ${\breve{p^*}:\calS \rightarrow \bfh(\wk{\calE})}$ is also left adjoint to ${\breve{p_!}:\bfh(\wk{\calE}) \rightarrow \calS}$ follows because  the canonical 
\[ (\bfh(\wk{\calE}))(p^* A, X) = p_!(X^{p^* A})  \rightarrow (p_! X)^A = \calS(A, p_! X) \]
is an iso for every weakly Kan object $X$ and $A$ in $\calS$.
\end{proof}

   In particular, if ${p:\calE \rightarrow \calS}$ is cohesive (i.e.\ satisfies Continuity) then the $\calS$-adjunction ${\breve{p_!} \dashv \breve{p^*}:\calS \rightarrow \bfh\calE}$ is `quintessential' in the enriched sense suggested by Lemma~\ref{LemWKisQuintessential}. This is the case considered in Theorem~1 of \cite{Lawvere07}, where the canonical ${\calE \rightarrow \bfh\calE}$ is called an {\em extensive quality}. Let us briefly discuss the terminology.
   
   Johnstone defines a quintessential localization as a string of adjoints ${I \dashv Q \dashv I:\calX \rightarrow \calY}$ with fully faithful ${I:\calX \rightarrow \calY}$.
   For convenience, we will say that ${Q:\calY \rightarrow \calX}$ is a quintessential localization. Notice that if ${p:\calE \rightarrow \calS}$ is a quality type (Section~\ref{SecPreCohesive}) then ${p_*:\calE\rightarrow \calS}$ is quintessential.
   
     Lawvere defines  `qualities' as certain functors which have, as domain, a cohesive category over a base $\calX$ and, as codomain, a quality type over the same base. In fact, Lawvere emphasizes two kinds of qualities that we recall below.
   
    An {\em intensive quality} on the pre-cohesive ${p:\calE \rightarrow \calS}$ is a functor ${s:\calE \rightarrow \calL}$ where ${q:\calL \rightarrow \calS}$ is a quality type, and satisfying that  ${s:\calE \rightarrow \calL}$ preserves finite products and finite coproducts and the following diagram 
$$\xymatrix{
\calE \ar[rd]_-{p_*} \ar[r]^-s & \calL \ar[d]^-q \\
 & \calS
}$$   
commutes. See Definition~4 in \cite{Lawvere07} and also Theorem~2 loc.\ cit..
   
    An {\em extensive quality} on the pre-cohesive ${p:\calE \rightarrow \calS}$ is a finite-coproduct preserving functor ${h:\calE \rightarrow \calL}$ where ${q:\calL \rightarrow \calS}$ is a quality type and such that the following diagram 
$$\xymatrix{
\calE \ar[rd]_-{p_!} \ar[r]^-s & \calL \ar[d]^-q \\
 & \calS
}$$   
commutes.

   One may interpret the definition of extensive quality in the context of $\calS$-enriched categories and it is in this sense that Lemma~\ref{LemWKisQuintessential} says that ${\calE \rightarrow \bfh(\wk\calE)}$ is an extensive quality. We insist that this is a refinement of Theorem~1 in \cite{Lawvere07} to a context a pre-cohesive ${p:\calE \rightarrow \calS}$ where Continuity need not hold.

   To summarize let us assume  that $\calE$ is a topos equipped with a sufficiently cohesive pre-cohesive geometric morphism ${p:\calE \rightarrow \calS}$ to another topos $\calS$. If the Continuity condition holds then we have a canonical `extensive quality' ${\calE \rightarrow \bfh\calE}$ (enriched over $\calS$) which assigns to each object in $\calE$ the associated `homotopy type' in the Hurewicz category of $\calE$. If Continuity does not hold then the situation is more subtle. We certainly have the $\calS$-functor ${\calE\rightarrow \bfh{\calE}}$ and the adjunction  ${\breve{p_!}\dashv \breve{p^*}:\calS \rightarrow \bfh\calE}$ but ${\breve{p_!}:\bfh\calE \rightarrow \calS}$ is not in general a quality type. On the other hand, we have a quality type ${\bfh(\wk\calE) \rightarrow \calS}$ but, as far as we can see, there is not an $\calS$-functor ${\calE \rightarrow \bfh(\wk\calE)}$ in general. In practice, though, there is evidence that the inclusion ${\bfh(\wk\calE) \rightarrow \bfh\calE}$ has a right adjoint. See Chapter~{IV.3} in \cite{GabrielZisman}.

\subsection{Preservation of pieces and Hurewicz categories}

    The purpose of this section is to prove that a geometric morphism that preserves pieces induces an enriched adjunction at the level of Hurewicz categories. Again, the idea of the proof may be sketched easily. For assume that we have  a geometric morphism ${g:\calF \rightarrow \calE}$ over a base topos $\calS$, from a pre-cohesive ${f:\calF \rightarrow \calS}$ to a pre-cohesive ${p:\calE \rightarrow \calS}$. If $g$ preserves pieces then the informal calculation
\[ (\bfh\calF)(g^* E, X) = f_!(X^{g^* E}) = p_! (g_*  (X^{g^* E}) ) = p_! ((g_* X)^E) = (\bfh\calE)(E, g_* X) \]
suggests that ${g:\calF\rightarrow \calE}$ may indeed induce an $\calS$-enriched adjunction 
${\bfh\calF \rightarrow \bfh\calE}$ between the associated Hurewicz categories.

   In order to give an actual proof it is better to start, as before, with a little bit of enriched category theory. Let $\calV$ and $\calW$ be cartesian closed categories and let ${F:\calV \rightarrow \calW}$ be a functor that preserves finite products.

\begin{lemma}\label{LemQuicoClaimsFolklore}
   If ${F:\calV \rightarrow \calW}$ has a finite-product preserving left adjoint then 
 ${\breve{F}:F_{\bullet} \calV \rightarrow \calW}$ has a $\calW$-enriched left adjoint.
\end{lemma} 
\begin{proof}
Let ${L:\calW \rightarrow \calV}$ be the left adjoint to $F$   (with unit $\nu$ and counit $\xi$). The left adjoint to $\breve{F}$ will be denoted by ${\breve{L}:\calW \rightarrow F_{\bullet} \calV}$. On objects it is simply $L$, whereas for $W$, $X$ in $\calW$, ${\breve{L}_{W, X}:\calW(W, X) \rightarrow (F_{\bullet} \calV)(\breve{L} W, \breve{L} X) = F(\calV(L W, L X))}$ is the map ${X^W \rightarrow F((L X)^{(L W)})}$ given by the composite
$$\xymatrix{
X^W \ar[r]^-{\nu} & F (L (X^W)) \ar[rr]^-{F(\kappa^L_{W, X})} && F((L X)^{(L W)})
}$$ 
in $\calW$. The unit of ${\breve{L}\dashv \breve{F}}$ is given by the family of maps ${\ulcorner\nu_{W}\urcorner: 1 \rightarrow (F (L W))^W}$ indexed by ${W \in \calW}$, where ${\ulcorner\nu_{W}\urcorner}$ is the exponential transpose of ${\nu_W:W \rightarrow F(L W)}$.
The counit is the family of maps
$$\xymatrix{1 \ar[r]^-{\cong} & F 1 \ar[rr]^-{F \ulcorner \xi_U \urcorner} && F(U^{L (F U)})}$$
where ${\ulcorner \xi_U \urcorner:1 \rightarrow U^{L (F U)}}$ is the exponential transpose of $\xi_U$.

It is well known (see 1.11 in \cite{kelly} for instance)
that  an adjunction as in the statement is equivalent to having an $\calE$-natural isomorphism
${F(U^{(L W)}) \rightarrow (F U)^W}$, and it is not difficult to see that the above isomorphism, in this case, is the composite
$$\xymatrix{
F(U^{(L W)}) \ar[rr]^-{\kappa^F_{L W, U}} && (F U)^{F(L W)} \ar[rr]^-{(F U)^{\nu}} && (F U)^W
}$$
with inverse
$$\xymatrix{
(F U)^W \ar[r]^-{\nu} & F (L ((F U)^W)) \ar[rr]^-{F (\kappa^L_{W, F U})} && F((L(F U))^{(L W)})
\ar[rr]^-{F(\xi^{(L W)})} && F(U^{(L W)})
}$$
in $\calW$.
\end{proof}

   Assume now that $\calS$, $\calF$ and $\calE$ are toposes, and that ${f:\calF \rightarrow \calS}$, ${p:\calE \rightarrow \calS}$ and ${g:\calF \rightarrow \calE}$ are geometric morphisms such that ${p g = f:\calF \rightarrow \calS}$. Assume further that $f$ and $p$ are pre-cohesive.

\begin{proposition}
If ${g:\calF \rightarrow \calE}$ preserves pieces then it induces  an $\calS$-enriched adjunction 
\[
\xy
(0,0)*+{\bfh\calF}="1";
(25,0)*+{\bfh\calE}="2";
\POS "1" \ar@<-.2cm>_-{\breve{g_*}}^-\bot "2",
\POS "2" \ar@<-.15cm>_-{\breve{g^*}} "1",
\endxy
\]
between the Hurewicz categories determined by $f\pts\calF\to \calS$ and $p\pts\calE\to\calS$.
If, moreover, ${f:\calF \rightarrow \calS}$ is cohesive then this adjunction restricts to one
\[
\xy
(0,0)*+{\bfh\calF}="1";
(25,0)*+{\bfh(\wk\calE)}="2";
\POS "1" \ar@<-.2cm>_-{\breve{g_*}}^-\bot "2",
\POS "2" \ar@<-.15cm>_-{\breve{g^*}} "1",
\endxy
\]
in $\enrichedIn{\calS}$.
\end{proposition}
\begin{proof}
 By Lemma~\ref{LemQuicoClaimsFolklore} we obtain an $\calE$-enriched adjunction
\[
\xy
(0,0)*+{(g_*)_\bullet(\calF)}="1";
(25,0)*+{\calE}="2";
\POS "1" \ar@<-.2cm>_-{\breve{g_*}}^-\bot "2",
\POS "2" \ar@<-.15cm>_-{\breve{g^*}} "1",
\endxy
\]
so if we now apply $(p_!)_\bullet\pts \calE{\textup{-\textbf{Cat}}}\to\calS{\textup{-\textbf{Cat}}}$ and use the iso
$\lambda\pts p_!g_*\to f_!$, we obtain the $\calS$-enriched adjunction
\[
\xy
(0,0)*+{(f_!)_\bullet \calF\simeq(p_!g_*)_\bullet(\calF)=(p_!)_\bullet((g_*)_\bullet \calF)}="1";
(60,0)*+{(p_!)_\bullet\calE}="2";
\POS "1" \ar@<-.2cm>_-{\breve{g_*}}^-\bot "2",
\POS "2" \ar@<-.15cm>_-{\breve{g^*}} "1",
\endxy
\]
in the statement.
   If ${f:\calF \rightarrow \calS}$ is cohesive then $g$ restricts to an adjunction ${g^* \dashv g_*:\calF \rightarrow \wk\calE}$ by Proposition~\ref{PropDirectImagePreservesWK}.
\end{proof}

   It is worth mentioning that we have not used the fact that ${g^*:\calE \rightarrow \calF}$ preserves equalizers. So we may consider pieces-preserving adjunctions whose left adjoints only preserves finite products. For example, if we let $q$ be the adjunction ${p_!\dashv p^*:\calS \rightarrow \calE}$ then $q$ preserves pieces when considered as below
\[
\xymatrix{
\calS \ar[rd]_-1 \ar[r]^-q & \calE \ar[d]^-p \\
 & \calS
}
\]
and we have the following corollary.

\begin{corollary}\label{fortameadjunction}
Assume that $p\pts\calE\to\calS$ satisfies the Nullstellensatz. There is an $\calS$-enriched adjunction
\[
\xy
(0,0)*+{\calS}="1";
(25,0)*+{\bfh(\calE)}="2";
\POS "1" \ar@<-.2cm>_-{\widehat{p^*}}^-\bot "2",
\POS "2" \ar@<-.15cm>_-{\widehat{p_!}} "1",
\endxy
\]
(which coincides with that at the beginning of Section~\ref{SecHurewicz}).
\end{corollary}

\section*{Acknowledgments} We thank Bill Lawvere for his help and encouragement. Both authors thank 
DGAPA-PAPIIT,  project No.\ IN110814, for financial support; in particular, this generous support  allowed the second 
author to visit the Instituto de Matem\'aticas UNAM in Mexico City, where an important part of the present work was done.
   We also want to thank the editor and referee for an efficient handling of the paper.  Finally, the second author wants to thank G.~Minian for responding to a number of bibliographic inquiries.

\bibliography{bibliopp}
\bibliographystyle{plain}

\end{document}